\PassOptionsToPackage{unicode}{hyperref}
\PassOptionsToPackage{naturalnames}{hyperref}
\documentclass[letterpaper,12pt]{article}
\usepackage[margin=1in]{geometry}
\usepackage[bookmarks, colorlinks=true, plainpages = false, colorlinks=true, citecolor=red, linkcolor=blue, anchorcolor=red, urlcolor=blue]{hyperref}
\usepackage{url}
\usepackage{amsmath, amsfonts, amsthm, amssymb, bm, bbm, verbatim, dsfont, mathtools, mathrsfs}
\usepackage{color, graphicx}
\usepackage{etoolbox}
\usepackage{almendra}
\usepackage{authblk}
\usepackage{array}
\usepackage{multirow}
\usepackage{ulem}
\usepackage{cancel}

\definecolor{ForestGreen}{RGB}{34,139,34}

\usepackage{float}
\usepackage{pgf,tikz}
\usetikzlibrary{arrows,shapes.arrows,shapes.geometric,shapes.multipart,fit,automata,decorations.pathmorphing,positioning}
\tikzset{
	>=stealth',
	true/.style={
		rectangle,
		draw=black, very thick,
		text width=6.5em,
		minimum height=2em,
		text centered,
		fill=gray, opacity = 0.5},
	punkt/.style={
		rectangle,
		rounded corners,
		draw=black, very thick,
		text width=6.5em,
		minimum height=2em,
		text centered},
	est/.style={
		circle,
		draw=black, very thick,
		text centered},
	shade/.style={
		circle,
		draw=black, very thick, fill=gray!50,
		text centered},
	weight/.style={
		circle,
		draw=black, very thick,
		text width=6.5em,
		minimum height=2em,
		text centered},
	pil/.style={
		->,
		thick,
		shorten <=2pt,
		shorten >=2pt,},
	double/.style={
		<->,
		thick,
		shorten <=2pt,
		shorten >=2pt,},
	dash/.style={
		dashed,
		thick,
		shorten <=2pt,
		shorten >=2pt,},
	dashdouble/.style={
		<->,
		dashed,
		thick,
		shorten <=2pt,
		shorten >=2pt,}
}

\usepackage{tikz-cd}
\makeatletter 
\newcolumntype{C}[1]{>{\centering\arraybackslash}p{#1}}

\usepackage{epstopdf}
\usepackage{fullpage}
\usepackage{algorithm}
\usepackage{algorithmicx}
\usepackage{subcaption}

\usepackage[nameinlink]{cleveref}
\usepackage[round]{natbib}

\usepackage[utf8]{inputenc} 
\usepackage[T1]{fontenc}    
\usepackage{url}            
\usepackage{booktabs}     
\usepackage{nicefrac}       
\usepackage{microtype}      
\usepackage{xcolor}
\usepackage{pdfpages}
\usepackage{scalerel}
\usepackage{dashrule}
\usepackage{lmodern}
\usepackage{fancyhdr}
\usepackage{tabu}
\usepackage{enumitem}

\def\IIFF{\mathbb{IF}}
\def\IF{\mathsf{IF}}

\def\var{\mathsf{var}}
\def\cov{\mathsf{cov}}

\def\Holder{\text{H\"{o}lder}}

\def\Bias{\mathsf{Bias}}

\def\se{\mathsf{se}}

\renewcommand{\hat}{\widehat}
\renewcommand{\tilde}{\widetilde}

\newcommand{\myreferences}{Master.bib}

\usepackage{xr,xspace}
\usepackage{todonotes}

\usepackage{caption,subcaption,soul}
\usepackage{algpseudocode}
\makeatletter
\usepackage{romanbar}

\theoremstyle{plain}
\newtheorem{theorem}{Theorem}
\newtheorem{lemma}{Lemma}
\newtheorem{proposition}{Proposition}
\newtheorem{corollary}{Corollary}

\theoremstyle{definition}
\newtheorem{definition}{Definition}
\newtheorem{example}{Example}
\AtBeginEnvironment{example}{%
  \pushQED{\qed}%
}
\AtEndEnvironment{example}{\popQED\endexample}

\newtheorem{assumption}{Assumption}

\newtheorem{problem}{Problem}

\newtheorem{remark}{Remark}
\AtBeginEnvironment{remark}{%
  \pushQED{\qed}%
}
\AtEndEnvironment{remark}{\popQED\endexample}

\newtheorem*{remark*}{Remark}
\usepackage[title]{appendix}

\usepackage{algorithm}
\usepackage{algpseudocode}

\usepackage{xspace, prettyref}

\newcommand{\Id}{\bbI}

\newcommand{\diff}{{\mathrm d}}

\newrefformat{eq}{(\ref{#1})}
\newrefformat{chap}{Chapter~\ref{#1}}
\newrefformat{sec}{Section~\ref{#1}}
\newrefformat{alg}{Algorithm~\ref{#1}}
\newrefformat{fig}{Fig.~\ref{#1}}
\newrefformat{tab}{Table~\ref{#1}}
\newrefformat{rmk}{Remark~\ref{#1}}
\newrefformat{clm}{Claim~\ref{#1}}
\newrefformat{def}{Definition~\ref{#1}}
\newrefformat{cor}{Corollary~\ref{#1}}
\newrefformat{lmm}{Lemma~\ref{#1}}
\newrefformat{prop}{Proposition~\ref{#1}}
\newrefformat{prob}{Problem~\ref{#1}}
\newrefformat{app}{Appendix~\ref{#1}}
\newrefformat{hyp}{Hypothesis~\ref{#1}}
\newrefformat{thm}{Theorem~\ref{#1}}

\newcommand{\tr}{\widetilde{r}}

\newcommand{\sfv}{{\mathsf{v}}}

\newcommand{\sfz}{{\mathsf{z}}}

\newcommand{\sfH}{{\mathsf{H}}}

\newcommand{\sfZ}{{\mathsf{Z}}}

\newcommand{\calA}{{\mathcal{A}}}
\newcommand{\calB}{{\mathcal{B}}}

\newcommand{\calE}{{\mathcal{E}}}

\newcommand{\calG}{{\mathcal{G}}}
\newcommand{\calH}{{\mathcal{H}}}

\newcommand{\calL}{{\mathcal{L}}}

\newcommand{\calN}{{\mathcal{N}}}
\newcommand{\calO}{{\mathcal{O}}}
\newcommand{\calP}{{\mathcal{P}}}

\newcommand{\calX}{{\mathcal{X}}}

\renewcommand{\tilde}{\widetilde}
\renewcommand{\hat}{\widehat}

\newcommand{\bbE}{{\mathbb{E}}}

\newcommand{\bbI}{{\mathbb{I}}}

\newcommand{\bbP}{{\mathbb{P}}}
\newcommand{\bbQ}{{\mathbb{Q}}}
\newcommand{\bbR}{{\mathbb{R}}}
\newcommand{\bbS}{{\mathbb{S}}}

\newcommand{\bbU}{{\mathbb{U}}}
\newcommand{\bbV}{{\mathbb{V}}}

\newcommand{\bbZ}{{\mathbb{Z}}}

\newcommand{\sfzbar}{\bar{\sfz}}
\newcommand{\sfvbar}{\bar{\sfv}}
\newcommand{\sfZbar}{\bar{\sfZ}}

\usepackage[many, most]{tcolorbox}

\def\Diag{\mathsf{Diag}}

\makeatletter
\def\ubar#1{\underline{\sbox\tw@{#1}\dp\tw@\z@\box\tw@}}
\makeatother

\def\Bias{\mathsf{bias}}
\def\csBias{\mathsf{cs\mbox{-}bias}}
\def\EB{\mathsf{kern\mbox{-}bias}}
\def\scrD{\mathscr{D}}
\def\tr{\mathrm{nuis}}
\def\cp{\mathfrak{c}}
\def\fM{\mathfrak{M}}

\def\mytitle{New $\sqrt{n}$-consistent, numerically stable higher-order influence function estimators}

\begin{document}

    \title{\mytitle}

    \author{Lin Liu\thanks{Correspondence: \href{linliu@sjtu.edu.cn}{linliu@sjtu.edu.cn}. The authors would like to thank \href{https://gaofn.xyz/}{Fengnan Gao}, \href{https://zhenyu-liao.github.io/}{Zhenyu Liao}, \href{https://scholar.harvard.edu/rajarshi/home}{Rajarshi Mukherjee}, \href{https://www.hsph.harvard.edu/james-robins/}{Jamie Robins}, and \href{https://sites.google.com/view/zheng-zhang}{Zheng Zhang} for invaluable discussions on this paper. The authors gratefully acknowledges funding support by NSFC Grant No.12101397 and No.12090024, Shanghai Municipal Science and Technology Grant No.2021SHZDZX0102, Shanghai Science and Technology Commission Grant No.21JC1402900, Shanghai Natural Science Foundation Grant No.21ZR1431000.}}\affil[2]{Institute of Natural Sciences, MOE-LSC, School of Mathematical Sciences, CMA-Shanghai, SJTU-Yale Joint Center for Biostatistics and Data Science, Shanghai Jiao Tong University and Shanghai Artificial Intelligence Laboratory, Shanghai, China}
    \author{Chang Li}\affil[1]{Department of Statistics, University of Virginia, Charlottesville, VA, USA}
    
    \date{\today}
    
    \maketitle
    
    \begin{abstract}
    Higher-Order Influence Functions (HOIFs) provide a unified theory for constructing rate-optimal estimators for a large class of low-dimensional (smooth) statistical functionals/parameters (and sometimes even infinite-dimensional functions) that arise in substantive fields including epidemiology, economics, and the social sciences. Since the introduction of HOIFs by \citet{robins2008higher} or \citet{robins2016technical}\footnote{\citet{robins2016technical} is the complete version of \citet{robins2008higher}, including more results and proofs. We therefore only refer to \citet{robins2016technical} in the sequel.}, they have been viewed mostly as a theoretical benchmark rather than a useful tool for statistical practice. Works aimed to flip the script are scant, but a few recent papers \citet{liu2017semiparametric, liu2021assumption} make some partial progress. In this paper, we take a fresh attempt at achieving this goal by constructing new, numerically stable HOIF estimators (or sHOIF estimators for short with ``s'' standing for ``stable'') with provable statistical and computational guarantees. This new class of sHOIF estimators (up to the 2nd order) was foreshadowed in synthetic experiments conducted by \citet{liu2020nearly}.
    \end{abstract}
    

    \newpage

    \section{Introduction}
    \label{sec:intro}
    
    {\it Higher-Order Influence Functions} (HOIFs) \citep{robins2016technical} are higher-order generalizations of the first-order influence functions (IFs), a staple in semiparametric statistical theory \citep{newey1990semiparametric, bickel1998efficient, van2002part}. HOIFs are a powerful and unified approach to constructing minimax rate-optimal estimators for a class of statistical functionals/parameters (and sometimes even functions; see \citet{kennedy2022minimax}) that arise in (bio)statistics, epidemiology, economics, and the social sciences. HOIF estimators originally proposed in \citet{robins2016technical, robins2017minimax}\footnote{See \citet{robins2022corrigenda} for corrections of the proofs in \citet{robins2017minimax}.} remain the only known minimax rate-optimal estimators for statistical functionals/parameters with substantive interests in the above disciplines, including the Average Treatment Effect (ATE) under the strong ignorability assumption\footnote{In \citet{liu2021assumption}, we derived the HOIFs for the ATE functional even when the strong ignorability assumption fails to hold, provided that we have access to valid proxies for both the treatment and outcome, following a series of works on proximal causal learning \citep{tchetgen2020introduction}.} and the expected conditional covariance of two random variables $A$ and $Y$ given a third random variable $X$, even after highly active research by the statistics and econometrics communities in recent years \citep{newey2018cross, kennedy2020optimal, hirshberg2021augmented, yu2020treatment}. More recent works \citep{kennedy2022minimax, bonvini2022fast} also initiated the application of the HOIF machinery to the minimax optimal estimation of Conditional Average Treatment Effect (CATE) function or dose response curves. Their results lay important theoretical foundation for individualized decision making problems, e.g. personalized medicine. This is the first instance when HOIF estimators are also shown to be effective, at least in theory, for function estimation problems, or more precisely, ``hybrid function and functional estimation problems''. Similar idea has also been applied to dose-response curve estimation \citep{bonvini2022fast}. For an introductory level review of HOIFs, we refer the interested readers to \citet{van2014higher} and Section 1 of \citet{liu2020rejoinder}. A relatively more technical review of HOIFs is delegated to Section \ref{sec:review}.
    
    Over the past decade, Robins and colleagues initiated the research program of establishing theoretical foundations for HOIFs and estimators based on HOIFs \citep{robins2004optimal, van2014higher, robins2016technical, robins2017minimax, liu2017semiparametric} for a class of statistical functionals/parameters recently characterized in \citet{rotnitzky2021characterization}, which are heretofore termed as {\it Doubly Robust Functionals} (DRF) in this paper. We adopt this terminology to reflect the fact that their nonparametric first-order IFs give rise to doubly robust estimators \citep{scharfstein1999adjusting, robins2001comments, chernozhukov2018double}. This class of DRFs subsumes the class of functionals studied in \cite{robins2016technical} and \cite{chernozhukov2018riesz}. Under the standard \Holder{}-regularity assumptions on the nuisance parameters (abbreviated as \Holder{} nuisance models), \citet{robins2017minimax} constructed minimax optimal but non-adaptive HOIF estimators for a sub-class of DRFs. \citet{liu2021adaptive} constructed adaptive second-order IF estimators for DRFs using the celebrated Lepski\v{i}'s adaptation scheme \citep{lepskii1991problem}, within a strict submodel of the \Holder{} nuisance models. But both estimators require estimating the density of the potentially high-dimensional covariates $X$, even in $\sqrt{n}$-estimable regimes. When the dimension $d$ of the covariates is only moderately large (e.g. $d = 10$), nonparametric density estimation is already a daunting computational and statistical task.
    
    To overcome the above issues, \citet{liu2017semiparametric} introduced {\it empirical} HOIF (eHOIF for short) estimators that obviate multi-dimensional density estimation by inverting the sample/empirical Gram matrix of vector-valued basis transformation of the covariates computed using a separate sample independent of the sample used to construct the estimator of the DRF. This sample-splitting strategy is adopted mainly for simplifying the mathematical analysis, leading to rather straightforward analysis of the statistical properties of the eHOIF estimators. In particular, the eHOIF estimators are still the only class of estimators that achieves $\sqrt{n}$-consistency and semiparametric efficiency for DRFs under the minimal \Holder{}-regularity assumptions \citep{robins2009semiparametric}. These nice statistical properties of the eHOIF estimators also motivate the development of a class of assumption-lean hypothesis tests statistic that is designed to falsify if the standard $(1 - \alpha) \times 100\%$ Wald confidence interval of the DRF has the claimed coverage probability \citep{liu2020nearly, liu2021assumption}. At this point, astute readers must wonder why we need a new class of empirical HOIF estimators at all, which is what this article is all about.
    
    \subsection{Motivation and main contributions}
    \label{sec:novelty}
    
    Despite the effort in \citet{liu2017semiparametric}, from our past experience of using eHOIF estimators in practice \citep{liu2017semiparametric, liu2020nearly, liu2021assumption, wanis2023machine}, several singular issues of their finite-sample performance were unveiled by large-scale simulation experiments\footnote{For interested readers, these simulation experiments have also been used to expose the gap between the (nonparametric) statistical theory deep neural networks (DNNs) and their practice in \citet{xu2022deepmed}. One can access computer codes of generating such simulations \href{https://github.com/siqixu/DeepMed}{here}.}:
    
    \begin{enumerate}[label = (\roman*)]
    \item \underline{Numerical instability}: In \citet{liu2017semiparametric}, although eHOIF estimators exhibit better finite-sample performance than the original HOIF estimators in \citet{robins2017minimax}, the simulations were restricted to very low condition number $k / n$: e.g. $k \approx 1,000$ and $n \approx 10,000$. In the simulation studies of \citet{liu2020nearly}, when $k$ gets near $n$, eHOIF estimators blow up numerically (see Section S3.1 of \citet{liu2020nearly}) already at order two. What is more striking is that the eHOIF estimators at higher orders, though supposed to be correcting the bias, can only exacerbate the numeric blow-up.
    \item \underline{Non-monotone bias reduction}: Theoretical results in \citet{liu2017semiparametric} hint that increasing the orders of the estimator should {\it in principle} reduce the bias. However, we found that this is not usually the case for eHOIF estimators in practice (e.g. see Section 5 of \citep{liu2021assumption}). Interestingly, sHOIF estimators do not seem to suffer from this problem in simulations, elevating the theoretical results from {\it mere principles} closer to {\it empirical facts}; see \citet{liu2020nearly} or \citet{wanis2023machine} for simulations at orders 2 or 3.
    \end{enumerate}
    
    Our contributions are three-fold.
    \begin{itemize}
    \item \underline{Methodology and practical relevance}: This article proposes a new class of numerically stable sHOIF estimators for DRFs, that overcomes the above two major limitations of eHOIF estimators. The stable Second-Order IF (SOIF) estimators first appeared in the simulation studies of \citet{liu2020nearly}, but their statistical properties remain elusive.
    \item \underline{Theory and the proof strategy}: Obtaining a deeper theoretical underpinning of this phenomenon mandates meticulous calculations rather than crude upper bounds. This is the critical technical innovation vis-\`{a}-vis other HOIF-related works. In particular, we intensively use the following proof techniques: leave-out analysis, matrix-valued Taylor expansion, and combinatorial calculations (i.e. corollaries of the binomial identity). The proof strategy developed in this paper may be of independent interest.
    \item \underline{Extensions of sHOIFs beyond ATE settings}: We also generalize sHOIF estimators to all the DRFs, allowing us to handle more structural parameters in the current causal inference (or econometrics) literature.
    \end{itemize}
    
    \subsection{Notation}
    \label{sec:notation}
    
    Before proceeding, we gather some frequently used notation throughout the paper. We denote the observed data random vector as $O \in \calO$, where $\calO$ is its corresponding sample space. Let $\sfzbar_{k} \coloneqq (z_{1}, \cdots, z_{k})^{\top}$ denote a collection of $k$ different functions, each of which has input domain $\calX$. Fix some $\theta' \in \Theta$. $\bbE_{\theta'}$, $\var_{\theta'}$, and $\cov_{\theta'}$ are, respectively, the expectation, variance, and covariance operators under the probability law $\bbP_{\theta'}$. For any measurable function $h: \calX \rightarrow \bbR$, let $\Vert h \Vert_{\infty} \coloneqq \mathrm{ess } \sup_{x \in \calX} h (x)$ and $\Vert h \Vert_{\theta', p} \coloneqq \left\{ \bbE_{\theta'} \left[ h (X)^{p} \right] \right\}^{1 / p}$ for any $p \geq 1$. We adopt standard (stochastic) asymptotic notation $\lesssim$, $\gtrsim$, $\asymp$, $\gg$, $\ll$, $o (\cdot)$, $\omega (\cdot)$, $O (\cdot)$, $\Omega (\cdot)$, $o_{\bbP_{\theta'}} (\cdot)$, $O_{\bbP_{\theta'}} (\cdot)$. For any real-valued vector $\bar{v}$ and any $q \in \bbR$, let $v^{q}$ be the element-wise $q$-th power of $v$. 
    
    Furthermore, define $\Sigma_{\theta'} \coloneqq \bbE_{\theta'} [Q]$, where $Q \coloneqq A \sfzbar_{k} (X) \sfzbar_{k} (X)^{\top}$, as the ($A$-weighted) population Gram matrix of $\sfzbar_{k} (X)$, till Section \ref{sec:drf}, in which we generalize all our results from $\psi (\theta) \coloneqq \bbE_{\theta} [Y (a = 1)]$\footnote{We use the potential outcome notation without introducing it, which will not affect the understanding of the main theme of this work.}, the mean of outcome $Y$ in the treated group under strong ignorability, to all members of the DRFs. Similarly, define $\hat{\Sigma} \coloneqq \bbP_{n} [Q_{k}] \equiv n^{-1} \sum_{i = 1}^{n} Q_{i}$ as the ($A$-weighted) sample Gram matrix of $\sfzbar_{k} (X)$, again till Section \ref{sec:drf}. Here $\bbP_{n} [\cdot]$ denotes the sample mean operator. 
    To further lighten the notation, we let $Q_{I} \coloneqq \sum_{i \in I} Q_{i}$ for any multi-index set $I \subseteq [n]$. For convenience, we also denote multi-index set $\{i_{1}, i_{2}, \cdots, i_{j}\} \subseteq [n]$ as $\bar{i}_{j}$ for $j \leq n$. $\Omega_{\theta'} \equiv \Sigma_{\theta'}^{-1}$ and $\hat{\Omega} \equiv \hat{\Sigma}^{-1}$, when they exist, are respectively the inverse of the population and sample Gram matrices. The kernels constructed from $\sfzbar_{k}$ are denoted as $K_{\theta', k} (x, x') \coloneqq \sfzbar_{k} (x)^{\top} \Omega_{\theta'} \sfzbar_{k} (x')$ and $\hat{K}_{k} (x, x') = \sfzbar_{k} (x)^{\top} \hat{\Omega} \sfzbar_{k} (x')$. Given a set of functions $\sfvbar_{k}: \calX \rightarrow \bbR^{k}$ and any $L_{2}$ function $h: \calO \rightarrow \bbR$, $\Pi [h | \sfvbar_{k}]$ denotes the linear projection operator of projecting $h$ onto the linear span of $\sfvbar_{k}$: formally,
    $$
    \Pi [h | \sfvbar_{k}] (x) \coloneqq \sfvbar_{k} (x)^{\top} \bbE [\sfvbar_{k} (X) \sfvbar_{k} (X)^{\top}]^{-1} \bbE [\sfvbar_{k} (X) h (O)].
    $$
    
    We use $\bbP_{\theta}$ to denote the true data generating law, unless stated otherwise. When the reference measure is the true law $\bbP_{\theta}$, we often drop the dependence on $\theta$: for example, we write $\Vert h \Vert_{p} \equiv \Vert h \Vert_{\theta, p}$, $\Sigma \equiv \Sigma_{\theta}$, $\Omega \equiv \Omega_{\theta}$ and $\bbP$, $\bbE$, $\var$, $\cov$ correspond to $\bbP_{\theta}$, $\bbE_{\theta}$, $\var_{\theta}$, $\cov_{\theta}$. Note that $\Omega$ should not be confused with the asymptotic notation $\Omega (\cdot)$ and this will be clear from the context. A statistic is said to be ``oracle'' whenever it depends on some part(s) of the unknown {\it true} data generating law $\bbP_{\theta}$ (such as $\Omega$); otherwise it is said to be ``feasible''. We also introduce $\Diag$ as the operator of extracting the diagonal elements of a matrix.
    
    Finally, let $\bbU_{n, m} [\cdot]$ denote the $m$-th order $U$-statistic operator: for any function $h: \bbR^{m} \rightarrow \bbR$
    \begin{align*}
        \bbU_{n, m} [h (O_{1}, \cdots, O_{m})] \coloneqq \frac{(n - m)!}{n!} \sum_{1 \leq i_{1} \neq \cdots \neq i_{m} \leq n} h (O_{i_{1}}, \cdots, O_{i_{m}}).
    \end{align*}
    When $m = 1$, $\bbU_{n, m} [\cdot]$ reduces to the sample mean operator $\bbP_{n} [\cdot]$. Similarly, let $\bbV_{n, m}$ be the corresponding $V$-statistic operator\footnote{Here we use the scaling $\frac{(n - m)!}{n!}$ instead of the more conventional $\frac{1}{n^{m}}$ for notational convenience.}:
    \begin{align*}
	\bbV_{n, m} [h (O_{1}, \cdots, O_{m})] \coloneqq \frac{(n - m)!}{n!} \sum_{i_{1} = 1}^{n} \cdots \sum_{i_{m} = 1}^{n} h (O_{i_{1}}, \cdots, O_{i_{m}}).
    \end{align*}
    Later in the paper, for $m \geq 2$, we will define ``oracle'' $m$-th order influence function estimators constructed using the dictionary $\sfzbar_{k}$, denoted as $\hat{\IIFF}_{m, m, k} (\Omega) \equiv \hat{\IIFF}_{m, m, k} = \bbU_{n, m} [\hat{\IF}_{m, m, k, \bar{i}_{m}}]$ with $U$-statistic kernel $\IF_{m, m, k, \bar{i}_{m}} \equiv \hat{\IF}_{m, m, k, \bar{i}_{m}} (\Omega)$. Its stable feasible version is denoted by $\hat{\IIFF}_{m, m, k} (\hat{\Omega})$ with the corresponding kernel $\hat{\IF}_{m, m, k, \bar{i}_{m}} (\hat{\Omega})$. 
    
    
    \subsection{The setup and a review of the theory of HOIFs}
    \label{sec:review}
    
    With the notation just introduced, we are poised to state the problem setup and briefly review the theory of HOIFs relevant for this paper, in particular the theory of eHOIFs. 

    Suppose that we are given $N$ i.i.d. observations $\{O_{i}\}_{i = 1}^{N} \sim \bbP_{\theta}$, where $\theta \in \Theta$ is the so-called nuisance parameter and $\Theta$ is its underlying parameter space. Let $\calP \coloneqq \left\{ \bbP_{\theta}: \theta \in \Theta \right\}$ be the space of data generating probability measures. Our primary interest is to estimate and draw statistical inference on a smooth statistical functional $\psi (\theta): \rightarrow \bbR$, in the sense of \citet{van1991differentiable}. We restrict $\psi (\theta)$ to be the DRFs defined in \citet{rotnitzky2021characterization}. As mentioned, our running example is $\psi (\theta) = \bbE_{\theta} [Y (a = 1)]$ the mean of an outcome $Y$ in the treated group $A = 1$. Here the observed data specializes to $O = (X, A, Y)$: respectively the $d$-dimensional covariates belonging to a compact subset $\calX \equiv [-B, B]^{d}$ of $\bbR^{d}$, the binary treatment assignment, and the bounded outcome variable. Under unconfoundedness assumption (that can be relaxed by using the HOIFs of $\psi (\theta)$ under the proximal causal inference setting \citep{liu2021assumption}), $\psi (\theta)$ can be identified by either of the two statistical functionals of the observed data distribution:
    \begin{equation}
    \label{fnl}
    \psi (\theta) \equiv \bbE \left[ A a (X) Y \right] \equiv \bbE [b (X)]
    \end{equation}
    where $a (x) \coloneqq \{\bbE [A | X = x]\}^{-1}$ and $b (x) \coloneqq \bbE [Y | X = x, A = 1]$ except Section \ref{sec:drf}. For this functional $\psi (\theta)$, the nuisance parameter is $\theta \equiv (a, b, g)$ where $g (x)$ is the probability density/mass function of the covariates $X$ conditional on $A = 1$. Hence the nuisance parameter space $\Theta = \calA \times \calB \times \calG$, where $\calA, \calB, \calG$ are, respectively, the space where $a, b, g$ lie. We further divide the whole $N$ data points into two parts: one with sample size $n$, called the estimation sample, and the other with sample size $N - n$, called the nuisance sample used to estimate the nuisance parameter $\theta$. Throughout this paper, we condition on the nuisance sample data by treating it or any quantity computed from it as fixed.

    For a smooth statistical functional $\psi (\theta): \Theta \rightarrow \bbR$ in the sense of \citet{van1991differentiable}, its first-order influence function $\IIFF_{1} (\theta)$ is a mean-zero first-order $U$-statistic satisfying the following functional equation
    \begin{align*}
    \left. \frac{\diff \psi (\theta_{t})}{\diff t} \right\vert_{t = 0} = \bbE \left[ \IIFF_{1} (\theta) \cdot \bbS_{1} \right]
    \end{align*}
    where $\bbP_{\theta_{t}}$ is any parametric submodels in $\{\bbP_{\theta}, \theta \in \Theta\}$, such that when $t = 0$, $\bbP_{\theta_{t}} \equiv \bbP_{\theta}$, the true data generating law, and $\bbS_{1}$ is its first-order score vector, as defined in \citet{waterman1996projected}; also see \citet{robins2016technical}. Here $\IIFF_{1} (\theta)$ has the following form \citep{robins1994estimation}: 
    \begin{equation}
    \label{if1}
    \IIFF_{1} (\theta) \equiv \frac{1}{n} \sum_{i = 1}^{n} \IF_{1, i} (\theta), \text{ where } \IF_{1} (\theta) = A a (X) (Y - b (X)) + b (X) - \psi (\theta).
    \end{equation}
    Typically, classical semiparametric theory \citep{newey1990semiparametric, bickel1998efficient} constructs semiparametric efficient first-order estimators $\hat{\psi}_{1}$ of $\psi (\theta)$ based on its first-order influence function follows:
    $$
    \hat{\psi}_{1} = \frac{1}{n} \sum_{i = 1}^{n} A_{i} \hat{a} (X_{i}) (Y_{i} - \hat{b} (X_{i})) + \hat{b} (X_{i})
    $$
    where $\hat{a}, \hat{b}$ are nuisance parameter estimates computed from the nuisance sample. In particular, $\hat{\psi}_{1}$ has bias
    \begin{equation}
    \label{first-order bias}
    \begin{split}
    \Bias (\hat{\psi}_{1}) & = \bbE [\hat{\psi}_{1} - \psi (\theta)] = \bbE [\IF_{1} (\hat{\theta}) - \IF_{1} (\theta)] \\
    & = \bbE \left[ \left( \frac{\hat{a} (X)}{a (X)} - 1 \right) (b (X) - \hat{b} (X)) \right].
    \end{split}
    \end{equation}
    Formally, $\Bias (\hat{\psi}_{1})$ is a {\it product of two nuisance estimation errors}\footnote{\citet{rotnitzky2021characterization} actually define the general class of statistical functionals that permit doubly-robust estimators based on this second-order bias property; see Section \ref{sec:drf}.}, and hence {\it doubly-robust} \citep{scharfstein1999adjusting}.
    
    Despite being doubly-robust, the veracity of inference based on first-order estimators like $\hat{\psi}_{1}$ may nonetheless be questionable when the nuisance parameter $\theta$ is of high complexity: e.g. functions with low smoothness or without sparsity. For example, when $a, b$ belong to \Holder{} functions with smoothness $s_{a}, s_{b}$ and $g$ arbitrarily complex, by far no first-order estimators are known to be $\sqrt{n}$-consistency for estimating $\psi (\theta)$ throughout the entire range 
    \begin{equation}
    \label{minimal}
    \{(s_{a}, s_{b}): (s_{a} + s_{b}) / 2 \geq d / 4\}
    \end{equation}
    but the eHOIF estimators of \citet{liu2017semiparametric} or the original HOIF estimators of \citet{robins2016technical} if additionally assuming $g$ to be \Holder{} with smoothness $s_{g} > 0$. In fact, \citet{robins2009semiparametric} also showed that \eqref{minimal} is the minimal condition for the existence of $\sqrt{n}$-consistent estimators of $\psi (\theta)$ under the \Holder{} nuisance modeling assumption. Outside \eqref{minimal}, $\psi (\theta)$ is non $\sqrt{n}$-estimable and the only known estimator with the optimal rate of convergence in minimax sense is again the HOIF estimator \citep{robins2016technical, robins2017minimax, robins2022corrigenda}. When restricting to highly smooth $g$, \citet{liu2021adaptive} construct minimax optimal and adaptive estimator of $\psi (\theta)$ by combining the HOIF estimators with the celebrated Lepskii's adaptation scheme \citep{lepskii1991problem}.

    This article is about the $\sqrt{n}$-estimable regime \eqref{minimal}, so we will focus our attention on the eHOIF estimators. First, we choose a set of $k$-dimensional functions $\sfzbar_{k} \equiv (z_{1}, \cdots, z_{k})^{\top}: \calX \rightarrow \bbR^{k}$ satisfying certain regularity conditions to be given later in Section \ref{sec:soif}. The Second-Order Influence Function (SOIF) estimator of $\psi (\theta)$ is the following second-order $U$-statistic:
    \begin{equation}
    \label{soif}
    \begin{split}
    & \hat{\psi}_{2, k} (\Omega) \coloneqq \hat{\psi}_{1} + \hat{\IIFF}_{2, 2, k} \text{ where } \hat{\IIFF}_{2, 2, k} \equiv \hat{\IIFF}_{2, 2, k} (\Omega) \coloneqq \bbU_{n, 2} \left[ \IF_{2, 2, k; 1, 2} \right]
    \end{split}
    \end{equation}
    and
    \begin{align*}
    \IF_{2, 2, k; 1, 2} & \equiv \IF_{2, 2, k; 1, 2} (\Omega) \coloneqq \left( A_{1} \hat{a} (X_{1}) - 1 \right) \sfzbar_{k} (X_{1})^{\top} \Omega \sfzbar_{k} (X_{2}) A_{2} (Y_{2} - \hat{b} (X_{2})) \\
    & \equiv \left( A_{1} \hat{a} (X_{1}) - 1 \right) K_{k} (X_{1}, X_{2}) A_{2} (Y_{2} - \hat{b} (X_{2})).
    \end{align*}
    Based on the definition of HOIFs \citep{robins2016technical}, $- \hat{\IIFF}_{2, 2, k}$ is in fact the SOIF of $\Bias (\hat{\psi}_{1})$\footnote{The difference in the signs in $\hat{\IIFF}_{2, 2, k}$ between here and \citet{robins2016technical} is non-essential.}. A more intuitively appealing explanation goes as follows: $- \hat{\IIFF}_{2, 2, k}$ is an unbiased estimator of the following quantity:
    \begin{equation}
    \label{bias_k}
    \begin{split}
    \Bias_{k} (\hat{\psi}_{1}) & = \bbE \left[ \left( \frac{\hat{a} (X)}{a (X)} - 1 \right) \sfzbar_{k} (X)^{\top} \right] \Omega \bbE \left[ A \sfzbar_{k} (X) (b (X) - \hat{b} (X)) \right] \\
    & = \bbE \left[ \left( \frac{\hat{a} (X_{1})}{a (X_{1})} - 1 \right) K_{k} (X_{1}, X_{2}) A_{2} (b (X_{2}) - \hat{b} (X_{2})) \right]
    \end{split}
    \end{equation}
    which is simply replacing the estimation errors $\hat{a} / a - 1$ and $b - \hat{b}$ in \eqref{first-order bias} by
    $$
    \Pi \left[ \left. \frac{\hat{a}}{a} - 1 \right\vert \sfzbar_{k} \right] \text{ and } \Pi \left[ \left. b - \hat{b} \right\vert A \sfzbar_{k} \right].
    $$
    Hence $\hat{\IIFF}_{2, 2, k}$ can be interpreted as a bias correction term that {\it partially} debiases $\Bias (\hat{\psi}_{1})$.

    However, evaluating $\Omega$ in practice relies on the knowledge of $g$, which is generally unknown to the analyst. The initial attempt by \citet{robins2016technical} and \citet{robins2017minimax} was to estimate $g$ from the nuisance sample by $\hat{g}$, leading to statistical properties affected by $g - \hat{g}$ and thus complexity-reducing assumptions on $\calG \ni g$. To completely resolve this reliance, \citet{liu2017semiparametric} choose to estimate $\Omega$ by its empirical analogue using the {\it nuisance sample}, denoted as $\hat{\Omega}_{\tr} = \hat{\Sigma}_{\tr}^{-1}$. The resulting estimated kernel is denoted as $\hat{K}_{k}^{\tr} (x, x')$, similar to $\hat{K}_{k}$ defined in Section \ref{sec:notation}. Then the empirical SOIF (eSOIF) estimator $\hat{\IIFF}_{2, 2, k} (\hat{\Omega}_{\tr})$ of $\Bias_{k} (\hat{\psi}_{1})$ is
    \begin{align*}
    \hat{\IIFF}_{2, 2, k} (\hat{\Omega}_{\tr}) \equiv \bbU_{n, 2} \left[ \IF_{2, 2, k; 1, 2} (\hat{\Omega}_{\tr}) \right],
    \end{align*}
    which, unlike $\hat{\IIFF}_{2, 2, k}$, incurs a kernel estimation bias
    \begin{align*}
    \bbE [\hat{\IIFF}_{2, 2, k} (\hat{\Omega}_{\tr}) - \hat{\IIFF}_{2, 2, k}] & = \bbE [(A_{1} \hat{a} (X_{1}) - 1) \sfzbar_{k} (X_{1})^{\top}] (\hat{\Omega}_{\tr} - \bbI) \bbE [\sfzbar_{k} (X_{2}) A_{2} (Y_{2} - \hat{b} (X_{2}))] \\
    & = \bbE [(A_{1} \hat{a} (X_{1}) - 1) (\hat{K}_{k}^{\tr} (X_{1}, X_{2}) - K_{k} (X_{1}, X_{2})) A_{2} (Y_{2} - \hat{b} (X_{2}))],
    \end{align*}
    shown to be of order at most $\sqrt{k \log k / n}$ in \citet{liu2017semiparametric}. To further reduce the kernel estimation bias, one can consider the following $m$-th order eHOIF estimator, which is an $m$-th order $U$-statistic:
    \begin{align*}
    & \hat{\IIFF}_{(2, 2) \rightarrow (m, m), k} (\hat{\Omega}_{\tr}) \coloneqq \sum_{j = 2}^{m} \hat{\IIFF}_{j, j, k} (\hat{\Omega}_{\tr}) \\
    \text{where } & \hat{\IIFF}_{j, j, k} (\hat{\Omega}_{\tr}) \coloneqq \bbU_{n, j} \left[ \hat{\IF}_{j, j, k; 1, \cdots, j} (\hat{\Omega}_{\tr}) \right]
    \end{align*}
    and
    \begin{align*}
    \hat{\IF}_{j, j, k; 1, \cdots, j} (\hat{\Omega}_{\tr}) = (-1)^{j} (A_{1} \hat{a} (X_{1}) - 1) \sfzbar_{k} (X_{1})^{\top} \hat{\Omega}_{\tr} \prod_{s = 3}^{j} \left\{ (Q_{s} - \hat{\Sigma}_{\tr}) \hat{\Omega}_{\tr} \right\} \sfzbar_{k} (X_{2}) A_{2} (Y_{2} - \hat{b} (X_{2})).
    \end{align*}
    \citet{liu2017semiparametric} showed that the kernel estimation bias of $\hat{\IIFF}_{(2, 2) \rightarrow (m, m), k} (\hat{\Omega}_{\tr})$ is of order at most $(k \log k / n)^{m / 2}$ and variance of order at most $1 / n \vee k / n^{2}$. Hence by taking $m \asymp \sqrt{\log n}$ and $k \asymp n / \log^{c} n$ for some absolute constant $c > 0$, we could estimate $\Bias_{k} (\hat{\psi}_{1})$ with essentially no bias without inflating the order of the variance of $\hat{\psi}_{1}$. Furthermore, under \Holder{} nuisance models on $\calA \times \calB$, \citet{liu2017semiparametric} demonstrate that the sHOIF estimator $\hat{\psi}_{1} + \hat{\IIFF}_{(2, 2) \rightarrow (m, m), k} (\hat{\Omega}_{\tr})$, with said choices of $m$ and $k$, is $\sqrt{n}$-consistent in \eqref{minimal} and semiparametric efficient in the interior of \eqref{minimal} under some additional mild assumptions. In this paper, the sHOIF estimators to be introduced in Section \ref{sec:shoif} simply replace $\hat{\Sigma}_{\tr}$ and $\hat{\Omega}_{\tr}$ in the eHOIF estimators by $\hat{\Sigma}$ and $\hat{\Omega}$, the empirical analogues of $\Sigma$ and $\Omega$ computed from the estimation sample. One can easily see that, due to the correlation between $\hat{\Omega}$ and the estimation sample, the analysis of the statistical properties of sHOIF estimators becomes significantly more challenging.
    
    \subsection{Plan} \label{sec:outline}
    The rest of the paper is organized as follows. Section \ref{sec:soif} defines the stable Second-Order IF (sSOIF) estimators and studies their statistical and numerical properties as a warm-up. Section \ref{sec:shoif} presents the full version of sHOIF estimators, together with their statistical, numerical, and computational properties. We then apply sHOIF estimators and their statistical properties to two concrete problems Section \ref{sec:applications}: one is to show that sHOIF estimators for $\psi (\theta)$ achieve semiparametric efficiency under the minimal conditions within the classical \Holder{} nuisance models; the other is to use sHOIF estimators to test if the nominal $(1 - \alpha) \times 100\%$ Wald confidence interval centered at the first-order DML estimator has the claimed coverage, a novel assumption-lean statistical procedure recently proposed in \citet{liu2020nearly}, and further developed in \citet{liu2021assumption}. To demonstrate the generality of sHOIF estimators, Section \ref{sec:extensions} extends results heretofore in several directions. 
    Finally, Section \ref{sec:discussion} concludes the paper and discusses several open problems and possible future directions. Appendix contains technical details that provide insights on the proof strategy. The remaining technical details are deferred to Supplementary Materials \citep{shoif_supp}.
    
    \section{Assumptions and warm-up: Stable second-order influence function estimators}
    \label{sec:soif}
    
    In this section, we disclose the main assumptions, accompanied with an illustration of the main results using the stable second-order influence function (sSOIF) estimator $\hat{\IIFF}_{2, 2, k} (\hat{\Omega})$ as a warm-up of what follows. 
    
    The assumptions below are imposed throughout the paper unless stated otherwise.
    
    \begin{assumption}[Conditions on initial first-step nuisance parameter estimates.]
    \label{cond:nuis}
    Nuisance parameter estimators $\hat{a}$ and $\hat{b}$ are attained from a separate independent frozen nuisance sample. For simplicity, we assume this sample to also have size $n$. $\hat{a}$ and $\hat{b}$ further satisfy the following properties until otherwise noticed:
    \begin{enumerate}[label = (\roman*)]
        \item $\Vert \hat{a} - a \Vert_{2} = o (1)$ and $\Vert \hat{b} - b \Vert_{2} = o (1)$, i.e. both nuisance parameter estimators are $L_{2}$-consistent;
        \item $\Vert a \Vert_{\infty}$, $\Vert b \Vert_{\infty}$, $\Vert \hat{a} \Vert_{\infty}$ and $\Vert \hat{b} \Vert_{\infty}$ are bounded by some absolute constant $B > 0$.
        \item In the case of $\psi (\theta) = \bbE [Y (1)]$ under strong ignorability, we additionally need $1 / a$ and $1 / \hat{a}$ to be bounded between $(c, 1 - c)$ for some absolute constant $0 < c < 0.5$.
    \end{enumerate}
    \end{assumption}
    
    \begin{assumption}[Conditions on $\sfzbar_{k}$ related quantities.]
    \label{cond:b}
    The following are assumed on the basis functions $\sfzbar_{k}$ and the corresponding (inverse) Gram matrices $\Sigma, \hat{\Sigma}, \Omega, \hat{\Omega}$ and projection kernels $K_{k}$ and $\hat{K}_{k}$:
    \begin{enumerate}[label = (\roman*)]
        \item There exists an absolute constant $B > 0$ such that $\sup_{x \in \calX} K_{k} (x, x) \leq B k$ and $\sup_{x \in \calX} \hat{K}_{k} (x, x) \leq B k$;
        \item Both $\Sigma$ and $\hat{\Sigma}$ have bounded spectra;
        \item The projection kernel satisfies the following $L_{\infty}$-stability condition: for any measurable function $h: \calX \rightarrow \bbR$, 
        \begin{equation}
            \label{l_inf_stability}
            \left\Vert \Pi \left[ h | \sfzbar_{k} \right] (\cdot) \right\Vert_{\infty} \lesssim \Vert h \Vert_{\infty}.
        \end{equation}
    \end{enumerate}
    \end{assumption}
    
    \begin{remark}[Comments on Assumptions \ref{cond:nuis} and \ref{cond:b}] \leavevmode
    \begin{enumerate}[label = (\roman*)]
    \item Given Assumption \ref{cond:b}(ii), there is no loss of generality by assuming $\Sigma \equiv \Omega \equiv \bbI$, the identity matrix of the same size as $\Sigma$ or $\Omega$. We make such a simplification throughout the paper, unless stated otherwise.
    
    \item The assumptions on the nuisance parameters and their estimators in Assumption \ref{cond:nuis} are quite mild. In particular, we do not assume $\hat{a}$, $\hat{b}$ converge to $a$, $b$ at any algebraic rate in $L_{2}$-norm. In fact, if content with $\sqrt{n}$-consistency instead of semiparametric efficiency, $\Vert \hat{a} - a \Vert_{2} = o (1)$ and $\Vert \hat{b} - b \Vert_{2} = o (1)$ can be even relaxed to $\Vert \hat{a} - a \Vert_{2} = O (1)$ and $\Vert \hat{b} - b \Vert_{2} = O (1)$; see \citet{liu2017semiparametric}.
    
    \item Assumption \ref{cond:b} on the dictionary $\sfzbar_{k}$ also appeared in \citet{robins2017minimax, liu2017semiparametric, liu2020nearly, liu2021assumption}; also see comments in \citet{liu2020rejoinder}. The $L_{\infty}$-stability condition (iii) have been established for Cohen-Daubechies-Vial wavelets, B-splines, and local polynomial partition series \citep{belloni2015some}. It is possible to relax such a condition to a high-probability version, which we decide not to further pursue in this paper.
    \end{enumerate}
    \end{remark}

    The following result on the sSOIF estimator $\hat{\IIFF}_{2, 2, k} (\hat{\Omega})$ is a special case of Theorem \ref{thm:properties} to be revealed in Section \ref{sec:shoif}.
    
    \begin{proposition}[Bias and variance bounds of $\hat{\IIFF}_{2, 2, k} (\hat{\Omega})$.]
    \label{prop:soif}
    Under Assumptions \ref{cond:nuis} -- \ref{cond:b}, with $k = o (n)$, one has the following:
    \begin{enumerate}[label = {\normalfont(\roman*)}]
        \item The kernel estimation bias of $\hat{\IIFF}_{2, 2, k} (\hat{\Omega})$ satisfies
        \begin{equation} \label{EB2}
            \begin{split}
            & \EB_{2, k} (\hat{\psi}_{1}) \coloneqq \bbE \left[ \hat{\IIFF}_{2, 2, k} (\hat{\Omega}) \right] - \Bias_{\theta, k} (\hat{\psi}_{1}) \equiv \bbE \left[ \hat{\IIFF}_{2, 2, k} (\hat{\Omega}) - \hat{\IIFF}_{2, 2, k} \right] \\
            & \lesssim \frac{k}{n} \left\{ \left\Vert \frac{\hat{a} - 1}{a} \right\Vert_{2} \Vert \hat{b} - b \Vert_{2} + \left\Vert \frac{\hat{a} - a}{a} \right\Vert_{2} \Vert \hat{b} - b \Vert_{2} + \left( \left\Vert \frac{\hat{a} - 1}{a} \right\Vert_{2} \left\Vert \hat{b} - b \right\Vert_{\infty} \wedge \left\Vert \frac{\hat{a} - 1}{a} \right\Vert_{\infty} \left\Vert \hat{b} - b \right\Vert_{2} \right) \right\}.
            \end{split}
        \end{equation}
        \item The variance of $\hat{\IIFF}_{2, 2, k} (\hat{\Omega})$ satisfies
        \begin{equation} \label{var2}
            \begin{split}
            \var \left[ \hat{\IIFF}_{2, 2, k} (\hat{\Omega}) \right] \lesssim \frac{1}{n} \left\{ \frac{k}{n} + \left( \left\Vert \frac{\hat{a} - 1}{a} \right\Vert_{2} \left\Vert \hat{b} - b \right\Vert_{\infty} \wedge \left\Vert \frac{\hat{a} - 1}{a} \right\Vert_{\infty} \left\Vert \hat{b} - b \right\Vert_{2} \right) \right\}.
            \end{split}
        \end{equation}
    \end{enumerate}
    \end{proposition}
    
    \begin{remark}
    The dependence on the condition number $k / n$ in the kernel estimation bias upper bound of the eSOIF estimator $\hat{\IIFF}_{2, 2, k} (\hat{\Omega}_{\tr})$ in \citet{liu2017semiparametric} ($\sqrt{k \log k / n}$) is worse than that of the sSOIF estimator $\hat{\IIFF}_{2, 2, k} (\hat{\Omega})$ reported here ($k / n$).
    \end{remark}

    \allowdisplaybreaks
    \subsection{Proof sketch of Proposition \ref{prop:soif}}
    \label{sec:proof_idea_soif}
    
    \subsubsection{Kernel estimation bias bound}
    \label{sec:proof_idea_soif_eb}
    
    $\EB_{2, k} (\hat{\psi}_{1})$ can be controlled by repeatedly using the matrix identity $(A - B)^{-1} - A^{-1} = - A^{-1} B (A - B)^{-1}$ with $A \equiv \hat{\Sigma}^{\dag} \coloneqq n^{-1} \sum_{i = 3}^{n} Q_{i}$ and $B = n^{-1} Q_{1, 2}$:
    \begin{align*}
        & \ \EB_{2, k} (\hat{\psi}_{1}) \\
        = & \ \bbE \left[ (A_{1} \hat{a} (X_{1}) - 1) \sfzbar_{k} (X_{1})^{\top} (\hat{\Omega}^{-1} - \Id) \sfzbar_{k} (X_{2}) A_{2} (Y_{2} - \hat{b} (X_{2})) \right] \\
        = & \ \sum_{j = 1}^{J - 1} \bbE \left[ (A_{1} \hat{a} (X_{1}) - 1) \sfzbar_{k} (X_{1})^{\top} \left( \Id - \hat{\Sigma}^{\dag} - \frac{Q_{1, 2}}{n} \right)^{j} \sfzbar_{k} (X_{2}) A_{2} (Y_{2} - \hat{b} (X_{2})) \right] \\
        & + \bbE \left[ (A_{1} \hat{a} (X_{1}) - 1) \sfzbar_{k} (X_{1})^{\top} \left( \Id - \hat{\Sigma}^{\dag} - \frac{Q_{1, 2}}{n} \right)^{J} \hat{\Omega} \sfzbar_{k} (X_{2}) A_{2} (Y_{2} - \hat{b} (X_{2})) \right].
    \end{align*}
    By choosing $J \asymp \log n$, the second term of the above display can be shown to be $o (n^{- 1 / 2})$.
    
    For the first term, we only look at $j = 1, 2$ in the main text and the remaining analysis is a special case of the proof of Theorem \ref{thm:properties} in Appendix \ref{app:main}.
    
    For $j = 1$, we have
    \allowdisplaybreaks
    \begin{align}
        & \ \bbE \left[ (A_{1} \hat{a} (X_{1}) - 1) \sfzbar_{k} (X_{1})^{\top} \left( \Id - \hat{\Sigma}^{\dag} - \frac{Q_{1, 2}}{n} \right) \sfzbar_{k} (X_{2}) A_{2} (Y_{2} - \hat{b} (X_{2})) \right] \notag \\
        = & \ \frac{2}{n} \bbE \left[ \left( \frac{\hat{a} (X_{1})}{a (X_{1})} - 1 \right) \sfzbar_{k} (X_{1})^{\top} \right] \bbE \left[ \sfzbar_{k} (X_{2}) (b (X_{2}) - \hat{b} (X_{2})) \right] \notag \\
        & - \frac{1}{n} \bbE \left[ \left( A_{1} \hat{a} (X_{1}) - 1 \right) \sfzbar_{k} (X_{1})^{\top} Q_{1, 2} \sfzbar_{k} (X_{2}) A_{2} \left( Y_{2} - \hat{b} (X_{2}) \right) \right] \label{bias_j1} \\
        = & \ O \left( \frac{1}{n} \right) - \frac{1}{n} \bbE \left[ \left( \frac{\hat{a} (X_{1}) - 1}{a (X_{1})} \right) \sfzbar_{k} (X_{1})^{\top} \sfzbar_{k} (X_{1}) \sfzbar_{k} (X_{1})^{\top} \right] \bbE \left[ \sfzbar_{k} (X_{2}) (b (X_{2}) - \hat{b} (X_{2})) \right] \notag \\
        & - \frac{1}{n} \bbE \left[ \left( \frac{\hat{a} (X_{1})}{a (X_{1})} - 1 \right) \sfzbar_{k} (X_{1})^{\top} \right] \bbE \left[ \sfzbar_{k} (X_{2}) \sfzbar_{k} (X_{2})^{\top} \sfzbar_{k} (X_{2}) (b (X_{2}) - \hat{b} (X_{2})) \right] \notag \\
        \lesssim & \ \frac{1}{n} + \frac{k}{n} \left\Vert \frac{\hat{a} - 1}{a} \right\Vert_{2} \Vert b - \hat{b} \Vert_{2} + \frac{k}{n} \left\Vert \frac{\hat{a}}{a} - 1 \right\Vert_{2} \Vert b - \hat{b} \Vert_{2} \notag
    \end{align}
    where the last line follows from triangle inequality, Cauchy-Schwarz inequality and Assumptions \ref{cond:nuis}, \ref{cond:b}(i) and \ref{cond:b}(ii).
    
    For $j = 2$, we have
    \begin{align}
        & \ \bbE \left[ (A_{1} \hat{a} (X_{1}) - 1) \sfzbar_{k} (X_{1})^{\top} \left( \Id - \hat{\Sigma}^{\dag} - \frac{Q_{1, 2}}{n} \right)^{2} \sfzbar_{k} (X_{2}) A_{2} (Y_{2} - \hat{b} (X_{2})) \right] \notag \\
        = & \ \bbE \left[ \left( \frac{\hat{a} (X_{1})}{a (X_{1})} - 1 \right) \sfzbar_{k} (X_{1})^{\top} \right] \bbE \left[ \left( \Id - \hat{\Sigma}^{\dag} \right)^{2} \right] \bbE \left[ \sfzbar_{k} (X_{2}) (b (X_{2}) - \hat{b} (X_{2})) \right] \notag \\
        & - \frac{2}{n^{2}} \bbE \left[ \left( A_{1} \hat{a} (X_{1}) - 1 \right) \sfzbar_{k} (X_{1})^{\top} Q_{1, 2} \sfzbar_{k} (X_{2}) A_{2} \left( Y_{2} - \hat{b} (X_{2}) \right) \right] \notag \\
        & + \frac{1}{n^{2}} \bbE \left[ \left( A_{1} \hat{a} (X_{1}) - 1 \right) \sfzbar_{k} (X_{1})^{\top} Q_{1, 2}^{2} \sfzbar_{k} (X_{2}) A_{2} \left( Y_{2} - \hat{b} (X_{2}) \right) \right] \label{bias_j2_1} \\
        \eqqcolon & \ (\mathrm{I}) + (\mathrm{II}) + (\mathrm{III}). \notag
    \end{align}
    Since $(\mathrm{II})$ is dominated by the term for $j = 1$, we only need to further analyze $(\mathrm{I})$ and $(\mathrm{III})$. $\mathrm{(III)}$ can be bounded by
    \begin{equation}
        \label{first-non-commutative}
        (\mathrm{III}) \lesssim \left( \frac{k}{n} \right)^{2} \left\{ \left\Vert \frac{\hat{a} - 1}{a} \right\Vert_{2} \Vert b - \hat{b} \Vert_{2} + \left\Vert \frac{\hat{a}}{a} - 1 \right\Vert_{2} \Vert b - \hat{b} \Vert_{2} + \left( \left\Vert \frac{\hat{a} - 1}{a} \right\Vert_{\infty} \Vert b - \hat{b} \Vert_{2} \right) \wedge \left( \left\Vert \frac{\hat{a} - 1}{a} \right\Vert_{2} \Vert b - \hat{b} \Vert_{\infty} \right) \right\}
    \end{equation}
    where the first two terms are due to the first three terms in the (non-commutative) expansion of
    \begin{equation}
        \label{q12_expand}
        Q_{1, 2}^{2} = Q_{1}^{2} + Q_{2}^{2} + Q_{1} Q_{2} + Q_{2} Q_{1},
    \end{equation}
    and the third term comes from the last term in the above expansion. The appearance of the estimation error in $L_{\infty}$-norm is due to the opposite order of sample points indexed by $1$ and $2$ between the ``meat'' $Q_{2} Q_{1}$ and the ``bread slices'' of the ``sandwich'' structure $\sfzbar_{k} (X_{1})^{\top} [\cdots] \sfzbar_{k} (X_{2})$. 
    
    For $(\mathrm{I})$, we need to expand $(\Id - \hat{\Sigma}^{\dag})^{2}$.
    \begin{align}
        (\mathrm{I}) = & \ \bbE \left[ \left( \frac{\hat{a} (X_{1})}{a (X_{1})} - 1 \right) \sfzbar_{k} (X_{1})^{\top} \right] \bbE \left[ \Id - 2 \hat{\Sigma}^{\dag} + \hat{\Sigma}^{\dag}{}^{2} \right] \bbE \left[ \sfzbar_{k} (X_{2}) (b (X_{2}) - \hat{b} (X_{2})) \right] \notag \\
        = & \ \left( \frac{4}{n} - 1 \right) \bbE \left[ \left( \frac{\hat{a} (X_{1})}{a (X_{1})} - 1 \right) \sfzbar_{k} (X_{1})^{\top} \right] \bbE \left[ \sfzbar_{k} (X_{2}) (b (X_{2}) - \hat{b} (X_{2})) \right] \notag \\
        & + \frac{(n - 2) (n - 3)}{n^{2}} \bbE \left[ \left( \frac{\hat{a} (X_{1})}{a (X_{1})} - 1 \right) \sfzbar_{k} (X_{1})^{\top} \right] \bbE \left[ \sfzbar_{k} (X_{2}) (b (X_{2}) - \hat{b} (X_{2})) \right] \notag \\
        & + \frac{n - 2}{n^{2}} \bbE \left[ \left( \frac{\hat{a} (X_{1})}{a (X_{1})} - 1 \right) \sfzbar_{k} (X_{1})^{\top} \right] \bbE \left[ Q_{3}^{2} \right] \bbE \left[ \sfzbar_{k} (X_{2}) (b (X_{2}) - \hat{b} (X_{2})) \right] \notag \\
        = & \ \left( \frac{6}{n^{2}} - \frac{1}{n} \right) \bbE \left[ \left( \frac{\hat{a} (X_{1})}{a (X_{1})} - 1 \right) \sfzbar_{k} (X_{1})^{\top} \right] \bbE \left[ \sfzbar_{k} (X_{2}) (b (X_{2}) - \hat{b} (X_{2})) \right] \notag \\
        & + \left( \frac{1}{n} - \frac{2}{n^{2}} \right) \bbE \left[ \left( \frac{\hat{a} (X_{1})}{a (X_{1})} - 1 \right) \sfzbar_{k} (X_{1})^{\top} \right] \bbE \left[ Q_{3}^{2} \right] \bbE \left[ \sfzbar_{k} (X_{2}) (b (X_{2}) - \hat{b} (X_{2})) \right]. \label{bias_j2_2}
    \end{align}
    It is straightforward to see the first term in the last equality of the above display is dominated by the term for $j = 1$, whereas the second term can be shown to be bounded by
    \begin{align*}
        \frac{k}{n} \left\Vert \frac{\hat{a}}{a} - 1 \right\Vert_{2} \Vert b - \hat{b} \Vert_{2}.
    \end{align*}
    
    Taken together, the terms for $j = 1$ and $j = 2$ give the desired bound for $\EB_{2, k} (\hat{\psi}_{1})$ in \eqref{EB2}. It remains to prove the terms for $j \geq 3$ are of smaller order, which is deferred to Appendix \ref{app:main_bias} for the general case. For $j \geq 3$, the corresponding term is of order
    \begin{align*}
        \left( \frac{k}{n} \right)^{j - 1} \left\Vert \frac{\hat{a}}{a} - 1 \right\Vert_{2} \Vert b - \hat{b} \Vert_{2}.
    \end{align*}
    
    \begin{remark}
    \label{rem:compare}
    Now is a perfect time to compare how the analysis of the kernel estimation bias of sSOIF differs from that of eSOIF of \citet{liu2017semiparametric}. The only difference between the eSOIF and sSOIF estimators are the samples used to estimate $\Omega$. Using the nuisance sample instead, the (conditional) kernel estimation bias of $\hat{\IIFF}_{2, 2, k} (\hat{\Omega}_{\tr})$ conditioning on the nuisance sample data is
    \begin{align*}
        & \ \bbE \left[ \hat{\IIFF}_{2, 2, k} (\hat{\Omega}_{\tr}) - \hat{\IIFF}_{2, 2, k} \right] \\
        = & \ \bbE \left[ (A_{1} \hat{a} (X_{1}) - 1) \sfzbar_{k} (X_{1})^{\top} \right] \left( \hat{\Omega}_{\tr} - \Id \right) \bbE \left[ A_{2} \sfzbar_{k} (X_{2}) (Y_{2} - \hat{b} (X_{2})) \right].
    \end{align*}
    From this, we can conclude
    \begin{align*}
        \left\vert \bbE \left[ \hat{\IIFF}_{2, 2, k} (\hat{\Omega}_{\tr}) - \hat{\IIFF}_{2, 2, k} \right] \right\vert \lesssim \left( \frac{k \log k}{n} \right)^{1 / 2} \left\Vert \frac{\hat{a}}{a} - 1 \right\Vert_{2} \Vert \hat{b} - b \Vert_{2}
    \end{align*}
    by using matrix Bernstein or Khintchine inequality \citep{rudelson1999random, bandeira2021matrix}; also see \citet{couillet2022random}. \citet{liu2017semiparametric} further show that
    \begin{align*}
    \left\vert \bbE \left[ \hat{\IIFF}_{(2, 2) \rightarrow (m, m), k} (\hat{\Omega}_{\tr}) - \hat{\IIFF}_{2, 2, k} \right] \right\vert \lesssim \left( \frac{k \log k}{n} \right)^{m / 2} \left\Vert \frac{\hat{a}}{a} - 1 \right\Vert_{2} \Vert \hat{b} - b \Vert_{2}.
    \end{align*}
    However, as pointed out in Section \ref{sec:novelty}, the finite-sample performance of eHOIF estimators is not well-reflected by these upper bounds, prompting the need of developing sHOIF estimators.
    \end{remark}
    
    \subsubsection{Variance bound}
    \label{sec:proof_idea_soif_var}
    
    The variance bound is technically involved. The missing steps can be found in Appendix \ref{app:soif_var}. The key step is to show
    \begin{equation}
        \label{var-bound-key}
        \bbE \left[ A_{1} \sfzbar_{k} (X_{1})^{\top} \hat{\Omega} \sfzbar_{k} (X_{2}) A_{3} \sfzbar_{k} (X_{3})^{\top} \hat{\Omega} \sfzbar_{k} (X_{4}) \right] - \left( \bbE \left[ A_{1} \sfzbar_{k} (X_{1})^{\top} \hat{\Omega} \sfzbar_{k} (X_{2}) \right] \right)^{2} = O \left( \frac{1}{n} \right).
    \end{equation}
    
    To prove \eqref{var-bound-key}, it is sufficient to exhibit
    \begin{equation}
        \label{var-bound-key-1}
        \begin{split}
        & \ \bbE \left[ A_{1} \sfzbar_{k} (X_{1})^{\top} \sfzbar_{k} (X_{2}) A_{3} \sfzbar_{k} (X_{3})^{\top} \left( \hat{\Omega} - \Id \right) \sfzbar_{k} (X_{4}) \right] \\
        & - \bbE \left[ A_{1} \sfzbar_{k} (X_{1})^{\top} \sfzbar_{k} (X_{2}) \right] \bbE \left[ A_{1} \sfzbar_{k} (X_{1})^{\top} \left( \hat{\Omega} - \Id \right) \sfzbar_{k} (X_{2}) \right] = O \left( \frac{1}{n} \right)
        \end{split}
    \end{equation}
    and
    \begin{equation}
        \label{var-bound-key-2}
        \begin{split}
        & \ \bbE \left[ A_{1} \sfzbar_{k} (X_{1})^{\top} \left( \hat{\Omega} - \Id \right) \sfzbar_{k} (X_{2}) A_{3} \sfzbar_{k} (X_{3})^{\top} \left( \hat{\Omega} - \Id \right) \sfzbar_{k} (X_{4}) \right] \\
        & - \bbE \left[ A_{1} \sfzbar_{k} (X_{1})^{\top} \left( \hat{\Omega} - \Id \right) \sfzbar_{k} (X_{2}) \right] \bbE \left[ A_{1} \sfzbar_{k} (X_{1})^{\top} \left( \hat{\Omega} - \Id \right) \sfzbar_{k} (X_{2}) \right] = O \left( \frac{1}{n} \right).
        \end{split}
    \end{equation}
    
    Recall $\hat{\Omega} = \hat{\Sigma}^{-1}$ and $\hat{\Sigma} = n^{-1} \sum_{i = 1}^{n} Q_{i}$. We introduce independent ``ghost copies'' $Q_{1}', Q_{2}'$ of $Q_{1}, Q_{2}$ and denote $\hat{\Omega}' = (\hat{\Sigma}')^{-1}$ and $\hat{\Sigma}' = n^{-1} \sum_{i = 3}^{n} Q_{i} + n^{-1} (Q_{1}' + Q_{2}')$. Then \eqref{var-bound-key-1} is equivalent to
    \begin{equation}
        \label{var-bound-key-1-1}
        \bbE \left[ A_{1} \sfzbar_{k} (X_{1})^{\top} \sfzbar_{k} (X_{2}) A_{3} \sfzbar_{k} (X_{3})^{\top} \left( \hat{\Omega} - \hat{\Omega}' \right) \sfzbar_{k} (X_{4}) \right] = O \left( \frac{1}{n} \right).
    \end{equation}
    
    Let $\bar{\Omega} \equiv \bar{\Sigma}$ and $\bar{\Sigma} \equiv n^{-1} \sum_{i = 5}^{n} Q_{i}$. Repeating the matrix identity $(A + B)^{-1} - A^{-1} = - A^{-1} B (A + B)^{-1}$ on \eqref{var-bound-key-1-1} by setting $A = \bar{\Sigma}$ and $B = n^{-1} (Q_{1, 2} + Q_{3, 4})$ or $B = n^{-1} (Q_{1, 2}' + Q_{3, 4})$, we have
    \begin{align*}
        & \ \bbE \left[ A_{1} \sfzbar_{k} (X_{1})^{\top} \sfzbar_{k} (X_{2}) A_{3} \sfzbar_{k} (X_{3})^{\top} \left( \hat{\Omega} - \hat{\Omega}' \right) \sfzbar_{k} (X_{4}) \right] \\
        = & \ \sum_{j = 1}^{J - 1} (-1)^{j} \bbE \left[ A_{1} \sfzbar_{k} (X_{1})^{\top} \sfzbar_{k} (X_{2}) A_{3} \sfzbar_{k} (X_{3})^{\top} \left\{ \left( \bar{\Omega} \frac{Q_{1, 2} + Q_{3, 4}}{n} \right)^{j} - \left( \bar{\Omega} \frac{Q_{1, 2}' + Q_{3, 4}}{n} \right)^{j} \right\} \bar{\Omega} \sfzbar_{k} (X_{4}) \right] \\
        & + (-1)^{J} \bbE \left[ A_{1} \sfzbar_{k} (X_{1})^{\top} \sfzbar_{k} (X_{2}) A_{3} \sfzbar_{k} (X_{3})^{\top} \left\{ \left( \bar{\Omega} \frac{Q_{1, 2} + Q_{3, 4}}{n} \right)^{J} - \left( \bar{\Omega} \frac{Q_{1, 2}' + Q_{3, 4}}{n} \right)^{J} \right\} \hat{\Omega} \sfzbar_{k} (X_{4}) \right].
    \end{align*}
    
    Let $J \asymp \log n$. The second term of the above display can be shown to be $o (1 / n)$. Proceeding to the first term, it is easy to see from Assumption \ref{cond:b}(iii) that the term corresponding to $j = 1$:
    \begin{align*}
        & \ \frac{1}{n} \bbE \left[ A_{1} \sfzbar_{k} (X_{1})^{\top} \sfzbar_{k} (X_{2}) A_{3} \sfzbar_{k} (X_{3})^{\top} \bar{\Omega} (Q_{1, 2} - Q_{1, 2}') \bar{\Omega} \sfzbar_{k} (X_{4}) \right] \\
        = & \ \frac{1}{n} \left( \begin{array}{c}
        \bbE \left\{ A_{1} \sfzbar_{k} (X_{1})^{\top} \bbE [\sfzbar_{k} (X_{2})] \bbE [A_{3} \sfzbar_{k} (X_{3})]^{\top} \bar{\Omega} \sfzbar_{k} (X_{1}) \sfzbar_{k} (X_{1})^{\top} \bar{\Omega} \bbE [\sfzbar_{k} (X_{4})] \right\} \\
        + \ \bbE \left\{ \bbE [A_{1} \sfzbar_{k} (X_{1})]^{\top} \sfzbar_{k} (X_{2}) \bbE [A_{3} \sfzbar_{k} (X_{3})]^{\top} \bar{\Omega} A_{2} \sfzbar_{k} (X_{2}) \sfzbar_{k} (X_{2})^{\top} \bar{\Omega} \bbE [\sfzbar_{k} (X_{4})] \right\} \\
        - \ \bbE \left\{ \bbE [A_{1} \sfzbar_{k} (X_{1})]^{\top} \sfzbar_{k} (X_{2}) \bbE [A_{3} \sfzbar_{k} (X_{3})]^{\top} \bar{\Omega} \bbE [Q_{1, 2}'] \bar{\Omega} \bbE [\sfzbar_{k} (X_{4})] \right\}
        \end{array} \right) = O \left( \frac{1}{n} \right).
    \end{align*}
    
    Similarly, the terms corresponding to $j \geq 2$ can be shown to be $O \left( \frac{1}{n} \left( \frac{2 k}{n} \right)^{j - 1} \right)$, a consequence of Lemma \ref{lem:nc} below. Note that the extra factor $2$ appears because there are $O (2^{j})$ terms in total by expanding out $\left( \frac{Q_{1, 2} + Q_{3, 4}}{n} \right)^{j}$ and $\left( \frac{Q_{1, 2}' + Q_{3, 4}}{n} \right)^{j}$.
    
    \begin{lemma}
    \label{lem:nc}
    Given a positive integer $j$. Given any pair of integers $j_{1} \geq 0, j_{2} > 0$ such that $j_{1} + j_{2} = j$, for any subset $\mathfrak{j} \subseteq \{0, 1\}^{j}$ of the $j$-dimensional Boolean hypercube with $\Vert \mathfrak{j} \Vert_{1} = j_{1}$, we have
    \begin{equation}
        \label{nc-good}
        \bbE \left[ \sfzbar_{k} (X_{1})^{\top} \sfzbar_{k} (X_{2}) \sfzbar_{k} (X_{3})^{\top} \left( \prod_{\ell = 1}^{j} Q_{1, 2}^{\mathfrak{j}_{\ell}} Q_{3, 4}^{(1 - \mathfrak{j}_{\ell})} \right) \sfzbar_{k} (X_{4}) \right] \lesssim k^{j - 1}.
    \end{equation}
    However, if $j_{1} = 0$ and $j_{2} = j$, we have
    \begin{equation}
        \label{nc-bad}
        \bbE \left[ \sfzbar_{k} (X_{1})^{\top} \sfzbar_{k} (X_{2}) \sfzbar_{k} (X_{3})^{\top} Q_{3, 4}^{j} \sfzbar_{k} (X_{4}) \right] \lesssim k^{j}.
    \end{equation}
    \end{lemma}
    
    The proof of Lemma \ref{lem:nc} can be found in Appendix \ref{app:nc}. Finally, we defer the proof of \eqref{var-bound-key-2} to the online supplements, which can be proved in a similar fashion.
    
    \subsection{Numerical stability and time complexity of $\hat{\IIFF}_{2, 2, k} (\hat{\Omega})$}
    \label{sec:numerical_soif}
    
    As discussed in Section \ref{sec:intro}, the key motivation for proposing sHOIF estimators is the numerical instability observed for eHOIF estimators. As a warm-up, we rigorously prove the numerical stability and calculate the time complexity of $\hat{\IIFF}_{2, 2, k} (\hat{\Omega})$ in this section. The reason why $\hat{\IIFF}_{2, 2, k} (\hat{\Omega})$ can be numerically unstable is that when $k$ is near $n$, it is highly likely $\lambda_{\min} (\hat{\Sigma}) \approx 0$ and hence $\lambda_{\max} (\hat{\Omega})$ is close to infinity. But:
    \begin{proposition} \label{prop:stability}
    $\hat{\IIFF}_{2, 2, k} (\hat{\Omega})$ does not depend on the eigenvalues of $\hat{\Omega}$.
    \end{proposition}
    
    For ease of exposition, in what follows we let
    \begin{itemize}
        \item $\sfZbar_{n, k} \coloneqq \left( \sfzbar_{k} (X_{1}), \cdots, \sfzbar_{k} (X_{n}) \right)^{\top}$ as the $n \times k$-matrix of the dictionary vectors for all $n$ samples;
        \item $\sfZbar_{n, k}^{A} \coloneqq \left( A_{1} \sfzbar_{k} (X_{1}), \cdots, A_{n} \sfzbar_{k} (X_{n}) \right)^{\top}$ as the $n \times k$-matrix of the $A$-weighted dictionary vectors for all $n$ samples;
        \item $\bm{\calE}_{n, a} (\hat{a}) \coloneqq \left( \calE_{a} (\hat{a}, O_{1}), \cdots, \calE_{a} (\hat{a}, O_{n}) \right)^{\top}$ and $\bm{\calE}_{n, b} (\hat{b}) \coloneqq \left( \calE_{b} (\hat{b}, O_{1}), \cdots, \calE_{b} (\hat{b}, O_{n}) \right)^{\top}$.
    \end{itemize}
    
    \begin{proof}
    We can rewrite $\hat{\IIFF}_{2, 2, k} (\hat{\Omega})$ as
    \begin{equation}
        \label{soif_v_stats}
        \hat{\IIFF}_{2, 2, k} (\hat{\Omega}) = \frac{1}{n - 1} \bm{\calE}_{n, a} (\hat{a})^{\top} \left[ \Id - \Diag \right] \left\{ \sfZbar_{n, k} \left( \sfZbar_{n, k}^{\top} \sfZbar_{n, k}^{A} \right)^{-1} \sfZbar_{n, k}^{A \top} \right\} \bm{\calE}_{n, b} (\hat{b}).
    \end{equation}
    Now apply Singular Value Decomposition (SVD) on the matrices $\sfZbar_{n, k}^{A}$:
    \begin{align*}
        \sfZbar_{n, k}^{A} = U^{A} \Diag (D^{A}) V^{A \top}.
    \end{align*}
    Then 
    \begin{align*}
        \sfZbar_{n, k} \left( \sfZbar_{n, k}^{\top} \sfZbar_{n, k}^{A} \right)^{-1} \sfZbar_{n, k}^{A \top} = \sfZbar_{n, k}^{A} \left( \sfZbar_{n, k}^{A \top} \sfZbar_{n, k}^{A} \right)^{-1} \sfZbar_{n, k}^{A \top} = U^{A} U^{A \top}.
    \end{align*}
    So
    \begin{align*}
        \eqref{soif_v_stats} = \frac{1}{n - 1} \bm{\calE}_{n, a} (\hat{a})^{\top} \left[ \Id - \Diag \right] \left\{ U^{A} U^{A \top} \right\} \bm{\calE}_{n, b} (\hat{b}),
    \end{align*}
    which is completely independent of the eigenvalues of $\hat{\Omega}$ ($(D^{A})^{2}$ up to constant).
    \end{proof}
    Hence it is not surprising that $\hat{\IIFF}_{2, 2, k} (\hat{\Omega})$ is numerically stable even when $k \rightarrow n$. 
    
    \begin{remark} 
    \label{rem:self normalization}
    In a sense, $\hat{\IIFF}_{2, 2, k} (\hat{\Omega})$ can be viewed as a {\it self-normalized} version of $\hat{\IIFF}_{2, 2, k}$. It is generally expected that self-normalized statistics could have better statistical properties than the non-self-normalized ones \citep{pena2008self}. However, whether the perspective of self-normalization is useful for establishing statistical properties of $\hat{\IIFF}_{2, 2, k} (\hat{\Omega})$ is still unclear to us and is worth pursuing as a research problem.
    \end{remark}
    
    Furthermore, not only does the alternative formula \eqref{soif_v_stats} of $\hat{\IIFF}_{2, 2, k} (\hat{\Omega})$ directly imply its numerical stability, but also it hints at the complexity of computing $\hat{\IIFF}_{2, 2, k} (\hat{\Omega})$. Barring the time complexity of SVD ($O (n k^{2})$), the time complexity of $\hat{\IIFF}_{2, 2, k} (\hat{\Omega})$ scales with $k n$ at a linear instead of a quadratic rate. This can be seen from \eqref{soif_v_stats}, in which only two vector-matrix products are involved, each taking $O (n k)$ operations. Thus we have
    \begin{proposition}
    The time complexity of computing $\hat{\IIFF}_{2, 2, k} (\hat{\Omega})$ is $O (n k^{2})$, dominated by that of SVD.
    \end{proposition}
    
    \begin{remark}
    Another alternative way of arriving at the above conclusion is to observe that the $U$-statistic kernel of $\hat{\IIFF}_{2, 2, k} (\hat{\Omega})$, denoted as $\hat{\IF}_{2, 2, k, \bar{i}_{2}} (\hat{\Omega})$, is separable, in the following sense: there exists a pair (but not necessarily a unique pair) of functions $h_{1}, h_{2}$ such that
    \begin{align*}
        \hat{\IF}_{2, 2, k, \bar{i}_{2}} (\hat{\Omega}) \equiv h_{1} (O_{i_{1}}, \hat{\Omega}) \cdot h_{2} (O_{i_{2}}, \hat{\Omega}).
    \end{align*}
    \end{remark}

    \section{The hierarchy of sHOIF estimators}
    \label{sec:shoif}
    
    As indicated in Section \ref{sec:review}, the $m$-th order sHOIF estimator $\hat{\IIFF}_{m, m, k} (\hat{\Omega})$ takes the same form as the $m$-th order eHOIF estimator $\hat{\IIFF}_{m, m, k} (\hat{\Omega}_{\tr})$, with the sole difference that $\hat{\Omega}_{\tr}$ is replaced by $\hat{\Omega}$. Formally, the $m$-th order sHOIF and the corresponding $m$-th order estimator of $\psi (\theta)$ read as follows:
    \begin{equation}
        \label{shoif}
        \begin{split}
            \hat{\IIFF}_{m, m, k} (\hat{\Omega}) & \coloneqq (-1)^{m} \bbU_{n, m} \left[ \calE_{a} (\hat{a}; O_{1}) \sfzbar_{k} (X_{1})^{\top} \hat{\Omega} \prod_{s = 3}^{m} \left\{ \left( Q_{s} - \hat{\Sigma} \right) \hat{\Omega} \right\} \sfzbar_{k} (X_{2}) \calE_{b} (\hat{b}; O_{2}) \right], \\
            \hat{\psi}_{m, k} (\hat{\Omega}) & \coloneqq \hat{\psi}_{1} + \sum_{j = 2}^{m} \hat{\IIFF}_{j, j, k} (\hat{\Omega}) \equiv \hat{\psi}_{1} + \hat{\IIFF}_{(2, 2) \rightarrow (m, m), k} (\hat{\Omega}).
        \end{split}
    \end{equation}

    \begin{remark} \label{rem:form}
        The above sHOIF statistics are the same as the eHOIF statistics except that $\hat{\Omega}$ is constructed from the estimation sample instead of the nuisance sample.
    \end{remark}
    
    In this section, we first explain heuristically why $\hat{\IIFF}_{m, m, k} (\hat{\Omega})$ is enough to correct for the kernel estimation bias (see Section \ref{sec:explain}), after which the statistical, numerical and computational properties of $\hat{\IIFF}_{m, m, k} (\hat{\Omega})$ are stated formally.
    
    \subsection{Heuristic explanation}
    \label{sec:explain}
    
    In what follows we explain heuristically why the kernel estimation bias $\hat{\IIFF}_{2, 2, k} (\hat{\Omega})$ can be further corrected by adding:
    \begin{align*}
    \hat{\IIFF}_{3, 3, k} (\hat{\Omega}) \coloneqq \frac{n - 2}{n} \bbU_{n, 3} \left[ - \ \calE_{a} (\hat{a}; O_{1}) \sfzbar_{k} (X_{1})^{\top} \hat{\Omega} \left( Q_{3} - \hat{\Sigma} \right) \hat{\Omega} \sfzbar_{k} (X_{2}) \calE_{b} (\hat{b}; O_{2}) \right]
    \end{align*}
    and
    \begin{align*}
    \hat{\IIFF}_{4, 4, k} (\hat{\Omega}) \coloneqq \frac{n - 3}{n} \bbU_{n, 4} \left[ (A_{1} \hat{a} (X_{1}) - 1) \sfzbar_{k} (X_{1})^{\top} \hat{\Omega} \left( Q_{3} - \hat{\Sigma} \right) \hat{\Omega} \left( Q_{4} - \hat{\Sigma} \right) \hat{\Omega} \sfzbar_{k} (X_{2}) A_{2} (Y_{2} - \hat{b} (X_{2})) \right].
    \end{align*}
    For short, we define $\hat{\IIFF}_{(2, 2) \rightarrow (3, 3), k} (\hat{\Omega}) \coloneqq \hat{\IIFF}_{2, 2, k} (\hat{\Omega}) + \hat{\IIFF}_{3, 3, k} (\hat{\Omega})$ and $\hat{\IIFF}_{(2, 2) \rightarrow (4, 4), k} (\hat{\Omega}) \coloneqq \sum_{j = 2}^{4} \hat{\IIFF}_{j, j, k} (\hat{\Omega})$. Also note that the majority of this section is written for mathematical rigor.
    
    Simple algebra gives
    \begin{align*}
        & \ \hat{\IIFF}_{3, 3, k} (\hat{\Omega}) \\
        \equiv & - \frac{n - 2}{n} \frac{(n - 3)!}{n!} \sum_{1 \leq i_{1} \neq i_{2} \neq i_{3} \leq n} \calE_{a} (\hat{a}; O_{i_{1}}) \sfzbar_{k} (X_{i_{1}})^{\top} \hat{\Omega} \left( Q_{i_{3}} - \hat{\Sigma} \right) \hat{\Omega} \sfzbar_{k} (X_{i_{2}}) \calE_{b} (\hat{b}; O_{i_{2}}) \\
        = & \ \frac{1}{n} \bbU_{n, 2} \left[ \calE_{a} (\hat{a}; O_{1}) \sfzbar_{k} (X_{1})^{\top} \hat{\Omega} Q_{1, 2} \hat{\Omega} \sfzbar_{k} (X_{2}) \calE_{b} (\hat{b}; O_{2}) \right] - \frac{2}{n} \hat{\IIFF}_{2, 2, k} (\hat{\Omega}) \\
        \approx & \ \frac{1}{n} \bbU_{n, 2} \left[ \calE_{a} (\hat{a}; O_{1}) \sfzbar_{k} (X_{1})^{\top} \hat{\Omega} Q_{1, 2} \hat{\Omega} \sfzbar_{k} (X_{2}) \calE_{b} (\hat{b}; O_{2}) \right] \eqqcolon \tilde{\hat{\IIFF}}_{3, 3, k} (\hat{\Omega})
    \end{align*}
    and
    \allowdisplaybreaks
    \begin{align*}
        & \ \hat{\IIFF}_{4, 4, k} (\hat{\Omega}) \\
        \equiv & \ \frac{n - 3}{n} \frac{(n - 4)!}{n!} \sum_{1 \leq i_{1} \neq i_{2} \neq i_{3} \neq i_{4} \leq n} \calE_{a} (\hat{a}; O_{i_{1}}) \sfzbar_{k} (X_{i_{1}})^{\top} \hat{\Omega} \prod_{s = 3}^{4} \left[ \left( Q_{i_{s}} - \hat{\Sigma} \right) \hat{\Omega} \right] \sfzbar_{k} (X_{i_{2}}) \calE_{b} (\hat{b}; O_{i_{2}}) \\
        = & \ \frac{1}{n (n - 2)} \bbU_{n, 2} \left[ \calE_{a} (\hat{a}; O_{1}) \sfzbar_{k} (X_{1})^{\top} \hat{\Omega} Q_{1, 2} \hat{\Omega} Q_{1, 2} \hat{\Omega} \sfzbar_{k} (X_{2}) \calE_{b} (\hat{b}; O_{2}) \right] \\
        & - \frac{1}{n} \bbU_{n, 3} \left[ \calE_{a} (\hat{a}; O_{1}) \sfzbar_{k} (X_{1})^{\top} \hat{\Omega} Q_{3} \hat{\Omega} Q_{3} \hat{\Omega} \sfzbar_{k} (X_{2}) \calE_{b} (\hat{b}; O_{2}) \right] \\
        & - \frac{1}{n - 2} \left( 6 \hat{\IIFF}_{3, 3, k} (\hat{\Omega}) - \left( 1 - \frac{6}{n} \right) \hat{\IIFF}_{2, 2, k} (\hat{\Omega}) \right) \\
        \approx & \ \frac{1}{n^{2}} \bbU_{n, 2} \left[ \calE_{a} (\hat{a}; O_{1}) \sfzbar_{k} (X_{1})^{\top} \hat{\Omega} Q_{1, 2} \hat{\Omega} Q_{1, 2} \hat{\Omega} \sfzbar_{k} (X_{2}) \calE_{b} (\hat{b}; O_{2}) \right] \\
        & - \frac{1}{n} \bbU_{n, 3} \left[ \calE_{a} (\hat{a}; O_{1}) \sfzbar_{k} (X_{1})^{\top} \hat{\Omega} Q_{3} \hat{\Omega} Q_{3} \hat{\Omega} \sfzbar_{k} (X_{2}) \calE_{b} (\hat{b}; O_{2}) \right] \\
        \eqqcolon & \ \tilde{\hat{\IIFF}}_{4, 4, k} (\hat{\Omega}).
    \end{align*}
    
    First, observe that the expectation of the oracle version of $\tilde{\hat{\IIFF}}_{3, 3, k} (\hat{\Omega})$
    \begin{align*}
        \bbE \left[ \tilde{\hat{\IIFF}}_{3, 3, k} \right] = \frac{1}{n} \bbE \left[ \calE_{a} (\hat{a}; O_{1}) \sfzbar_{k} (X_{1})^{\top} Q_{1, 2} \sfzbar_{k} (X_{2}) \calE_{b} (\hat{b}; O_{2}) \right]
    \end{align*}
    exactly cancels \eqref{bias_j1}, the leading-order part of the kernel estimation bias of $\hat{\IIFF}_{2, 2, k} (\hat{\Omega})$ corresponding to $j = 1$.
    
    Next, observe that the expectation of the oracle version of $\tilde{\hat{\IIFF}}_{4, 4, k} (\hat{\Omega})$ is
    \begin{align*}
        \bbE \left[ \tilde{\hat{\IIFF}}_{4, 4, k} \right] = & \ \frac{1}{n^{2}} \bbE \left[ \calE_{a} (\hat{a}; O_{1}) \sfzbar_{k} (X_{1})^{\top} Q_{1, 2}^{2} \sfzbar_{k} (X_{2}) \calE_{b} (\hat{b}; O_{2}) \right] \\
        & - \frac{1}{n} \bbE \left[ \calE_{a} (\hat{a}; O_{1}) \sfzbar_{k} (X_{1})^{\top} \right] \bbE [Q_{3}^{2}] \bbE \left[ \sfzbar_{k} (X_{2}) \calE_{b} (\hat{b}; O_{2}) \right]
    \end{align*}
    which again cancels the kernel estimation bias of $\hat{\IIFF}_{(2, 2) \rightarrow (3, 3), k} (\hat{\Omega})$ truncated at $j = 2$, dominated by
    \begin{equation}
        \label{toif_eb}
        - \frac{1}{n^{2}} \bbE \left[ \calE_{a} (\hat{a}; O_{1}) \sfzbar_{k} (X_{1})^{\top} Q_{1, 2}^{2} \sfzbar_{k} (X_{2}) \calE_{b} (\hat{b}; O_{2}) \right] + \frac{1}{n} \bbE \left[ \calE_{a} (\hat{a}; O_{1}) \sfzbar_{k} (X_{1})^{\top} \right] \bbE [Q_{3}^{2}] \bbE \left[ \sfzbar_{k} (X_{2}) \calE_{b} (\hat{b}; O_{2}) \right]
    \end{equation}
    which can be derived from \eqref{bias_j2_1}, \eqref{bias_j2_2} and the kernel estimation bias of $\hat{\IIFF}_{3, 3, k} (\hat{\Omega})$ truncated at level $j = 1$; see Appendix \ref{app:toif_eb} for a more detailed calculation. Hence $\hat{\IIFF}_{(3, 3) \rightarrow (4, 4), k} (\hat{\Omega})$ further reduces the kernel estimation bias of $\hat{\IIFF}_{2, 2, k} (\hat{\Omega})$. 
    
    \subsection{Characterization of the bias and variance of the sHOIF estimators}
    \label{sec:properties}
    
    We now state the main theoretical result of this paper.
    \begin{theorem}
    \label{thm:properties}
    Under Assumptions \ref{cond:nuis} -- \ref{cond:b}, with $k \lesssim \frac{n}{\log^{2} n}$ and $m \gtrsim \sqrt{\log n}$, one has the following:
    \begin{enumerate}[label = {\normalfont(\roman*)}]
    \item The kernel estimation bias of $\hat{\IIFF}_{(2, 2) \rightarrow (m, m), k} (\hat{\Omega})$ satisfies
    \begin{equation} \label{EBm}
        \begin{split}
        & \EB_{m, k} (\hat{\psi}_{1}) \coloneqq \bbE \left[ \hat{\IIFF}_{(2, 2) \rightarrow (m, m), k} (\hat{\Omega}) \right] - \Bias_{\theta, k} (\hat{\psi}_{1}) \equiv \bbE \left[ \hat{\IIFF}_{(2, 2) \rightarrow (m, m), k} (\hat{\Omega}) - \hat{\IIFF}_{2, 2, k} \right] \\
        & \lesssim \left( \frac{k m}{n} \right)^{\lceil \frac{\lceil \frac{m - 1}{2} \rceil - 1}{2} \rceil \vee 1} \left\{ \begin{array}{c}
        \left\Vert \dfrac{\hat{a} - 1}{a} \right\Vert_{2} \Vert \hat{b} - b \Vert_{2} + \left\Vert \dfrac{\hat{a} - a}{a} \right\Vert_{2} \Vert \hat{b} - b \Vert_{2} \\ 
        + \left( \left\Vert \dfrac{\hat{a} - 1}{a} \right\Vert_{2} \left\Vert \hat{b} - b \right\Vert_{\infty} \wedge \left\Vert \dfrac{\hat{a} - 1}{a} \right\Vert_{\infty} \left\Vert \hat{b} - b \right\Vert_{2} \right)
        \end{array} \right\}.
        \end{split}
    \end{equation}
    \item For $m \geq 2$, the variance of $\hat{\IIFF}_{m, m, k} (\hat{\Omega})$ satisfies
    \begin{equation} \label{var_if_m}
        \begin{split}
        \var \left[ \hat{\IIFF}_{m, m, k} (\hat{\Omega}) \right] \lesssim \frac{1}{n} \left\{ \frac{k}{n} + \left( \left\Vert \frac{\hat{a} - 1}{a} \right\Vert_{2} \left\Vert \hat{b} - b \right\Vert_{\infty} \wedge \left\Vert \frac{\hat{a} - 1}{a} \right\Vert_{2} \left\Vert \hat{b} - b \right\Vert_{\infty} \right) \right\}.
        \end{split}
    \end{equation}
    And thus the variance of $\hat{\IIFF}_{(2, 2) \rightarrow (m, m), k} (\hat{\Omega})$ satisfies
    \begin{equation} \label{var_if_2_m}
        \begin{split}
        \var \left[ \hat{\IIFF}_{(2, 2) \rightarrow (m, m), k} (\hat{\Omega}) \right] \lesssim \frac{1}{n} \left\{ \frac{k}{n} + \left( \left\Vert \frac{\hat{a} - 1}{a} \right\Vert_{2} \left\Vert \hat{b} - b \right\Vert_{\infty} \wedge \left\Vert \frac{\hat{a} - 1}{a} \right\Vert_{2} \left\Vert \hat{b} - b \right\Vert_{\infty} \right) \right\}.
        \end{split}
    \end{equation}
    \end{enumerate}
    \end{theorem}
    
    The proof of the above theorem can be found in Appendix \ref{app:main} (for kernel estimation bias bound) and the online supplements (for variance bound).
    
    \begin{remark}[Asymptotic normality and the bootstrap approximation]
    As shown in \citet{liu2020nearly}, the asymptotic normality of the oracle statistic $\frac{\hat{\IIFF}_{2, 2, k} - \Bias_{\theta, k} (\hat{\psi}_{1})}{\se_{\theta} (\hat{\IIFF}_{2, 2, k})}$ follows from Theorem 1 of \citet{bhattacharya1992class} whence $1 \ll k \ll n^{2}$. Thus to show CLT of $\hat{\IIFF}_{(2, 2) \rightarrow (m, m), k} (\hat{\Omega})$ for any $m \geq 2$, it is sufficient to demonstrate under what conditions $\EB_{m, k} (\hat{\Omega}) \ll \se_{\theta} (\hat{\IIFF}_{2, 2, k})$. Bootstrap approximation (and its rate) of the distribution of $\hat{\IIFF}_{(2, 2) \rightarrow (m, m), k} (\hat{\Omega})$, or even of $\hat{\IIFF}_{2, 2, k}$ is still an important open problem, though \citet{liu2021assumption} have made some partial progress. A more thorough study of the conditions under which central limit theorem (CLT) or bootstrap approximation holds is beyond the scope of this paper.
    \end{remark}
    
    \subsection{Numerical stability and time complexity of sHOIF estimators}
    \label{sec:numerical}
	
    In what follows we consider the numerical and computational properties of sHOIF estimators, which extends the results in Section \ref{sec:numerical_soif} to higher-order. The first result in this section, Theorem \ref{thm:stability}, earmarks the ``stability'' of sHOIF estimators in terms of their independence of the eigenvalues of $\hat{\Omega}$, the root cause of the instability of eHOIF estimators.
	
    \begin{theorem} \label{thm:stability}
    $\hat{\IIFF}_{m, m, k} (\hat{\Omega})$ does not depend on the eigenvalues of $\hat{\Omega}$.
    \end{theorem}
	
    \begin{proof}
    The proof resembles the proof of Proposition \ref{prop:stability} closely by realizing that, for any $i, j \in [n]$,
    \begin{align*}
	A_{i} \sfzbar_{k} (X_{i})^{\top} \hat{\Omega} \sfzbar_{k} (X_{j}) = U^{A}_{i, \bullet} U^{A \top} U U_{j, \bullet},
    \end{align*}
    which is completely independent of the eigenvalues of $\hat{\Omega}$.
    \end{proof}
    Hence sHOIF estimators do not suffer from any numerical instability resulted from the large condition number of the sample Gram matrix when we let $k$ near $n$ in practice.

    \begin{remark}
    \label{rem:algorithm}
    Theorem \ref{thm:stability} also suggests a better way to compute sHOIF estimators. Instead of computing the sample Gram matrix $\hat{\Sigma}$ and its inverse $\hat{\Omega}$ using numerical methods, we should instead perform SVD on the basis matrices $\sfZbar_{n, k}$ and $\sfZbar_{n, k}^{A}$ and then compute $\hat{\IIFF}_{m, m, k} (\hat{\Omega})$. In fact, the upcoming R package \citep{wanis2023machine} for computing HOIF related statistics exactly uses this strategy.
    \end{remark}

    Since sHOIF estimators are numerically stable and thus are potentially useful tools for statistical practice \citep{liu2020nearly, wanis2023machine}, it is worth discussing the computational complexity of sHOIF estimators for general order $m$ as well.
	
    \begin{theorem} \label{thm:time}
    The time complexity of computing $\hat{\IIFF}_{m, m, k} (\hat{\Omega})$ is $O (\max \{(n k)^{\lceil (m - 1) / 2 \rceil}, n k^{2}\})$.
    \end{theorem}
	
    \begin{proof}
    Similar to the proof of Proposition \ref{prop:stability}, we need to rewrite $\hat{\IIFF}_{(2, 2) \rightarrow (m, m), k} (\hat{\Omega})$ in the form of a linear combination of $V$-statistics. Without loss of generality, we take $\calE_{a} (\hat{a}; O) \equiv \calE_{b} (\hat{b}; O) \equiv 1$. But let us first represent $\hat{\IIFF}_{(2, 2) \rightarrow (m, m), k} (\hat{\Omega})$ as the following series:
    \begin{align*}
        \hat{\IIFF}_{(2, 2) \rightarrow (m, m), k} (\hat{\Omega}) \equiv \sum_{j = 1}^{m} (-1)^{j} \binom{m - 1}{j - 1} \bbU_{n, j} \left[ \sfzbar_{k} (X_{1})^{\top} \hat{\Omega} \cdot \prod_{s = 3}^{m} \left( Q_{s} \hat{\Omega} \right) \cdot \sfzbar_{k} (X_{2}) \right].
    \end{align*}
 
    Note that the number of summations in $\hat{\IIFF}_{m, m, k} (\hat{\Omega})$ is
    \begin{align*}
	n (n - 1) \cdots (n - m + 1) = \sum_{j = 1}^{m} (-1)^{m - j} s (m, j) n^{j}
    \end{align*}
    where $s (m, j)$ are unsigned Stirling numbers of the first kind, or the number of permutations on $m$ elements with $j$ cycles. Accordingly one can write an $m$-th order $U$-statistic into a linear combination of $V$-statistics from order $1$ to order $m$, with the number of $j$-th order $V$-statistics, for $j = 1, \cdots, m$, equal to $s (m, j)$. 
	    
    The proof is completed by leveraging the special structure of the $U$-statistic kernel for sHOIF estimators.
    \end{proof}
	
    \begin{remark}
    \label{rem:comp-stat}
    Considering Theorem \ref{thm:properties} and Theorem \ref{thm:time} in tandem, there is a clear statistical-computational trade-off. 
    \noindent However, whether or not such statistical-computational trade-off is an emanation of possibly intrinsic computational hardness of estimating certain smooth statistical functionals is still an open problem

    Finally, we briefly comment on our philosophical stance on the usefulness of sHOIF estimators. sHOIF estimators are effectively infinite-order $U$-statistics, so given the current computing devices, there is no doubt that practitioners are not using sHOIF estimators in practice in near term. This is ``conditional'' on the availability of hardware. The numerical stability or lack thereof, however, is an issue  regardless of the availability of more powerful computing resources.
    \end{remark}

    \begin{remark}
    \label{rem:ehoif-comp}
    Theorem \ref{thm:time} also applies to eHOIF estimators \citep{liu2017semiparametric} and the original HOIF estimators of \citet{robins2016technical}, that needs an estimate of the density of the covariates $X$, if the time for density estimation is not counted.
    \end{remark}
    
    \section{Applications of the statistical properties of sHOIF estimators}
    \label{sec:applications}
    
    \subsection{Semiparametric efficiency under minimal \Holder{} assumptions on the nuisance functions}
    \label{sec:semi}
    
    In nonparametric statistics, the optimality of a statistical procedure is often evaluated under the \Holder{} nuisance models. 
    
    The above calculations culminate into the following theorem, which is the second main result of this paper.
    \begin{theorem}
    \label{thm:semi}
    If $\calA \times \calB \subseteq \calH (s_{a}, \calX) \times \calH (s_{b}, \calX)$ with $(s_{a} + s_{b}) / 2 \geq d / 4$, and choosing $m \asymp \sqrt{\log n}$ and $k \lesssim n / \log (n)^{2}$,
    \begin{equation}
        \sqrt{n} \left( \hat{\psi}_{m, k} - \psi (\theta) \right) \overset{\calL}{\rightarrow} \calN (0, \bbE [\IF_{1} (\theta)^{2}])
    \end{equation}
    where $\bbE [\IF_{1} (\theta)^{2}]$ is the semiparametric efficiency bound of $\psi (\theta)$.
    \end{theorem}
    
    \begin{remark}
    According to the lower bound of \citet{robins2009semiparametric} under the \Holder{} nuisance model, $(s_{a} + s_{b}) / 2 \geq d / 4$ is the minimal condition for the existence of a semiparametric efficient estimator of $\psi (\theta)$. It is not unreasonable to expect that this minimal condition also holds for most, if not all, of the DRFs.
    \end{remark}
        
    \subsection{Implications on the assumption-free bias testing procedure of \citet{liu2020nearly} and \citet{liu2021assumption}}
    \label{sec:bias_test}
    
    In light of the growing interest in understanding the performance of deep-learning-based causal inference \citep{farrell2021deep, chen2020causal} and the gap between these theoretical results and empirical performance \citep{xu2022deepmed}, \citet{liu2020nearly} proposed the following oracle assumption-free valid nominal $\alpha$-level test statistic:
    \begin{equation}
        \label{oracle_test}
        \hat{\chi} \coloneqq \mathbbm{1} \left\{ \frac{\hat{\IIFF}_{2, 2, k}}{\hat{\se} [\hat{\psi}_{1}]} - z_{\alpha / 2} \frac{\hat{\se} [\hat{\IIFF}_{2, 2, k}]}{\hat{\se} [\hat{\psi}_{1}]} > \delta \right\}
    \end{equation}
    for the following null hypothesis:
    \begin{equation}
        \label{h0}
        \sfH_{0} (\delta): \frac{\csBias (\hat{\psi}_{1})}{\se [\hat{\psi}_{1}]} \leq \delta
    \end{equation}
    where 
    \begin{equation}
        \csBias (\hat{\psi}_{1}) \coloneqq \left\{ \bbE \left[ \lambda (X) (\hat{a} (X) - a (X))^{2} \right] \bbE \left[ \lambda (X) (\hat{b} (X) - b (X))^{2} \right] \right\}^{1 / 2}.
    \end{equation}
    
    \citet{liu2021assumption} in turn constructed a feasible assumption-lean valid nominal $\alpha$-level test statistic 
    \begin{equation}
        \hat{\chi}_{3, k} (\hat{\Omega}_{\tr}) \coloneqq \mathbbm{1} \left\{ \frac{\hat{\IIFF}_{(2, 2) \rightarrow (3, 3), k} (\hat{\Omega}_{\tr})}{\hat{\se} [\hat{\psi}_{1}]} - z_{\alpha / 2} \frac{\hat{\se} [\hat{\IIFF}_{(2, 2) \rightarrow (3, 3), k} (\hat{\Omega}_{\tr})]}{\hat{\se} [\hat{\psi}_{1}]} > \delta \right\}
    \end{equation}
    and the following higher-order test statistic based on eHOIF estimators:
    \begin{equation}
    \hat{\chi}_{m, k} (\hat{\Omega}_{\tr}) \coloneqq \mathbbm{1} \left\{ \frac{\hat{\IIFF}_{(2, 2) \rightarrow (m, m), k} (\hat{\Omega}_{\tr})}{\hat{\se} [\hat{\psi}_{1}]} - z_{\alpha / 2} \frac{\hat{\se} [\hat{\IIFF}_{(2, 2) \rightarrow (m, m), k} (\hat{\Omega}_{\tr})]}{\hat{\se} [\hat{\psi}_{1}]} > \delta \right\}.
    \end{equation}
    
    \citet{liu2021assumption} showed that all the standard errors in the above test statistics can be estimated consistently by certain bootstrapping procedure. More importantly, they proved the following.
    \begin{proposition}
    Let
    \begin{align*}
        \csBias_{k} (\hat{\psi}_{1}) \coloneqq \left\{ \bbE \left[ \Pi [(A \hat{a} - 1) | \sfzbar_{k}] (X)^{2} \right] \bbE \left[ \Pi [A (\hat{b} - b) | \sfzbar_{k}] (X)^{2} \right] \right\}^{1 / 2}.
    \end{align*}
    Under the assumptions of Theorem \ref{thm:semi}, with $k \lesssim n / (\log n)^{2}$ and the following extra condition:
    \begin{equation}
        \label{extra}
        \vert \EB_{3, k} (\hat{\psi}_{1}) \vert \not\ll \csBias_{k} (\hat{\psi}_{1})
    \end{equation}
    then $\hat{\chi}_{3, k} (\hat{\Omega}_{tr})$ is a valid nominal $\alpha$-level test of $\sfH_{0} (\delta)$ \eqref{h0}. The extra condition \eqref{extra} can be relaxed to
    \begin{equation}
        \label{extra_relax}
        \vert \EB_{m, k} (\hat{\psi}_{1}) \vert \not\ll \csBias_{k} (\hat{\psi}_{1}) \left( \frac{k \log k}{n} \right)^{\frac{m - 1}{2}}
    \end{equation}
    if one uses $\hat{\chi}_{m, k} (\hat{\Omega}_{tr})$ instead of $\hat{\chi}_{3, k} (\hat{\Omega}_{tr})$.
    \end{proposition}
    
    We can similarly define the following sHOIF-based test statistics: for $m \geq 2$,
    \begin{align*}
        \hat{\chi}_{m, k} (\hat{\Omega}) \coloneqq \mathbbm{1} \left\{ \frac{\hat{\IIFF}_{(2, 2) \rightarrow (m, m), k} (\hat{\Omega})}{\hat{\se} [\hat{\psi}_{1}]} - z_{\alpha / 2} \frac{\hat{\se} [\hat{\IIFF}_{(2, 2) \rightarrow (m, m), k} (\hat{\Omega})]}{\hat{\se} [\hat{\psi}_{1}]} > \delta \right\}.
    \end{align*}
    
    Then as an immediate corollary of Theorem \ref{thm:properties}, we have
    \begin{theorem}
    Under the assumptions of Theorem \ref{thm:semi}, with $k \lesssim n / (\log n)^{2}$ and a different relaxed extra condition from \eqref{extra_relax}:
    \begin{equation}
        \label{extra_relax_stable}
        \vert \EB_{m, k} (\hat{\psi}_{1}) \vert \not\ll \csBias_{k} (\hat{\psi}_{1}) \left( \frac{k}{n} \right)^{m - 1}
    \end{equation}
    then $\hat{\chi}_{m, k} (\hat{\Omega})$ is a valid nominal $\alpha$-level test of $\sfH_{0} (\delta)$ \eqref{h0}.
    \end{theorem}
    
    Given the above theoretical guarantees, and further considering that the sHOIF estimators and tests have better finite-sample performance than the corresponding eHOIF estimators and tests, we recommend using $\hat{\chi}_{m, k} (\hat{\Omega})$ in practice. For more examples of its application, see \citet{wanis2023machine}.
    
    \section{Further extensions of sHOIF estimators}
    \label{sec:extensions}
    
    \subsection{Generalization to the entire class of DRFs}
    \label{sec:drf}
    
    In this subsection, we briefly comment on how our results can be generalized to the entire class of DRFs characterized in \citet{rotnitzky2021characterization}. The class of DRFs includes many other functionals that arise in substantive studies in (bio)statistics, epidemiology, economics, and social sciences, including:
    \begin{itemize}
        \item the expected conditional variance, which is useful for constructing confidence/predictive sets \citep{robins2006adaptive};
        \item the expected conditional covariance, which is useful for both causal inference and conditional independence testing \citep{shah2020hardness};
        \item average causal effect of continuous treatment, which is important for treatment allocations \citep{ai2021unified, bonvini2022fast}.
    \end{itemize}
    
    \citet{rotnitzky2021characterization} defined the class of DRFs as follows:
    \begin{definition}[Doubly Robust Functionals; Definition 1 of \citet{rotnitzky2021characterization}]
    $\psi (\theta)$ is a doubly robust functional if, for each $\theta \in \Theta$ there exists $a: x \mapsto a (x) \in \calA$ and $b: x \mapsto b (x) \in \calB$ such that (i) $\theta = (a, b, \theta \setminus \{b, p\})$ and $\Theta = \calA \times \calB \times \Theta \setminus \{\calA, \calB\}$ and (ii) for any $\theta, \theta' \in \Theta$ 
    \begin{equation}
    \psi (\theta) - \psi (\theta^{\prime}) + \bbE \left[ \IF_{1} (\theta^{\prime}) \right] = \bbE \left[ S (a (X) - a^{\prime} (X)) (b (X) - b^{\prime} (X)) \right] \label{eq:drbias}
    \end{equation}
    where $S \equiv s (O)$ with $o \mapsto s (o)$ a known function that does not depend on $\theta$ or $\theta'$ satisfying either $\bbP_{\theta} (S \geq 0) = 1$ or $\bbP_{\theta} (S \leq 0) = 1$. We also denote $\lambda (x) \coloneqq \bbE [S | X = x]$. Then the first-order influence function of $\psi (\theta)$ has the following form: given $\theta' \equiv (a', b', \theta' \setminus \{a', b'\})^{\top} \in \Theta$,
    \begin{equation}
        \label{drf_if1}
        \IF_{1} (\theta') \equiv S a' (X) b' (X) + m_{a} (O, a') + m_{b} (O, b') + S_{0}
    \end{equation}
    where $S_{0}$ is some known statistic that does not depend on $a'$ and $b'$, and $h \mapsto m_{a} (o, h)$ for $h \in \calA$ and $h \mapsto m_{b} (o, h)$ for $h \in \calB$ are two known linear maps satisfying
    \begin{align*}
        & \bbE \left[ S h (X) b (X) + m_{a} (O, h) \right] \equiv 0, \ \forall \ h \in \calA, \\
        & \bbE \left[ S a (X) h (X) + m_{b} (O, h) \right] \equiv 0, \ \forall \ h \in \calB.
    \end{align*}
    As a result, $\psi (\theta) \equiv \bbE [m_{a} (O, a) + S_{0}] \equiv \bbE [m_{b} (O, b) + S_{0}]$.
    \end{definition}
    
    \begin{remark}
    For $\psi (\theta) \equiv \bbE [Y (1)]$ under strong ignorability, $S$, $a$, $b$, $m_{a} (O, a)$, $m_{b} (O, b)$, and $S_{0}$ correspond to $- A$, $\{\bbE [A | X]\}^{-1}$, $\bbE [Y | X, A = 1]$, $A Y a (X)$, $b (X)$, and $0$, respectively. Thus $\lambda (X) = \bbE [- A | X] = - \frac{1}{a (X)}$. We also have
        \begin{align*}
            & \bbE [\calE_{a} (\hat{a}; O) \sfzbar_{k} (X)] = \bbE \left[ \frac{1}{a (X)} (\hat{a} (X) - a (X)) \sfzbar_{k} (X) \right], \\
            & \bbE [\calE_{b} (\hat{b}; O) \sfzbar_{k} (X)] = \bbE \left[ \frac{1}{a (X)} (\hat{b} (X) - b (X)) \sfzbar_{k} (X) \right].
        \end{align*}
    \end{remark}
    
    We have the following notation correspondence that maps the results for $\psi (\theta) \equiv \bbE [Y (1)]$ under strong ignorability to any DRF $\psi (\theta)$:
    \begin{itemize}
        \item $\calE_{a} (\hat{a}; O) \sfzbar_{k} (X) \Rightarrow \calE_{a} (\hat{a}, \sfzbar_{k}; O)$ and $\calE_{b} (\hat{b}; O) \sfzbar_{k} (X) \Rightarrow \calE_{b} (\hat{b}, \sfzbar_{k}; O)$ where $\calE_{a} (\hat{a}, \sfzbar_{k}; O)$ and $\calE_{b} (\hat{b}, \sfzbar_{k}; O)$ satisfy
        \begin{align*}
            & \bbE [\calE_{a} (\hat{a}, \sfzbar_{k}; O)] = \bbE \left[ \lambda (X) (\hat{a} (X) - a (X)) \sfzbar_{k} (X) \right], \\
            & \bbE [\calE_{b} (\hat{b}, \sfzbar_{k}; O)] = \bbE \left[ \lambda (X) (\hat{b} (X) - b (X)) \sfzbar_{k} (X) \right].
        \end{align*}
        \item $\Sigma = \bbE [A \sfzbar_{k} (X) \sfzbar_{k} (X)^{\top}]$ $\Rightarrow$ $\Sigma = \bbE [S \sfzbar_{k} (X) \sfzbar_{k} (X)^{\top}]$ and $\hat{\Sigma} = \bbP_{n} [A \sfzbar_{k} (X) \sfzbar_{k} (X)^{\top}]$ $\Rightarrow$ $\hat{\Sigma} = \bbP_{n} [S \sfzbar_{k} (X) \sfzbar_{k} (X)^{\top}]$.
    \end{itemize}
    With the above mappings, all the theoretical results for $\psi (\theta) \equiv \bbE [Y (1)]$ developed herein can be applied to those for an arbitrary DRF $\psi (\theta)$ {\it mutatis mutandis}.

    \subsubsection{A special case: the expected conditional covariance}
    \label{sec:special}

    Before concluding our paper, we further study the implications of the sHOIF theory developed so far for a special cases of DRFs: the expected conditional covariance between two random variables $A$ and $Y$ given a third random variable $X$, $\psi \coloneqq \bbE [\cov (A, Y | X)]$. When $A = Y$ almost surely, $\psi$ reduces to the expected conditional variance of $A$ given $X$, $\psi \coloneqq \bbE [\cov (A | X)]$. For differences between these two parameters, see an extended discussion in \citet{liu2020nearly}.

    The main feature that distinguishes $\psi$ from many other DRFs is $S = 1$, which leads to its SOIF:
    \begin{align*}
        \hat{\IIFF}_{2, 2, k} = \frac{1}{n (n - 1)} \sum_{1 \leq i_{1} \neq i_{2} \leq n} (A_{i_{1}} - \hat{a} (X_{i_{1}})) \sfzbar_{k} (X_{i_{1}})^{\top} \bar{\Omega} \sfzbar_{k} (X_{i_{2}}) (Y_{i_{2}} - \hat{b} (X_{i_{2}})),
    \end{align*}
    in which $\bar{\Omega} \equiv \{\bbE [S \sfzbar_{k} (X) \sfzbar_{k} (X)^{\top}]\}^{-1} \equiv \{\bbE [\sfzbar_{k} (X) \sfzbar_{k} (X)^{\top}]\}^{-1}$ only depends on the distribution of $X$. Also, $\hat{a}$ and $\hat{b}$ are nuisance estimates of $a (x) = \bbE [A | X = x]$ and $b (x) = \bbE [Y | X = x]$ in this context. This leads to the following improved kernel estimation bias bound:
    \begin{corollary}
    Under Assumptions \ref{cond:nuis} -- \ref{cond:b}, with $k \lesssim \frac{n}{\log^{2} n}$ and $m \gtrsim \sqrt{\log n}$, one has the following:
    
    \noindent The kernel estimation bias of $\hat{\IIFF}_{(2, 2) \rightarrow (m, m), k} (\hat{\Omega})$ satisfies
    \begin{equation} \label{special EBm}
        \begin{split}
        & \EB_{m, k} (\hat{\psi}_{1}) \coloneqq \bbE \left[ \hat{\IIFF}_{(2, 2) \rightarrow (m, m), k} (\hat{\Omega}) \right] - \Bias_{\theta, k} (\hat{\psi}_{1}) \equiv \bbE \left[ \hat{\IIFF}_{(2, 2) \rightarrow (m, m), k} (\hat{\Omega}) - \hat{\IIFF}_{2, 2, k} \right] \\
        & \lesssim \left( \frac{k m}{n} \right)^{\lceil \frac{\lceil \frac{m - 1}{2} \rceil - 1}{2} \rceil \vee 1} \left\{ \left\Vert \hat{a} - a \right\Vert_{2} \Vert \hat{b} - b \Vert_{2} + \left( \left\Vert \hat{a} - a \right\Vert_{2} \left\Vert \hat{b} - b \right\Vert_{\infty} \wedge \left\Vert \hat{a} - a \right\Vert_{\infty} \left\Vert \hat{b} - b \right\Vert_{2} \right) \right\}.
        \end{split}
    \end{equation}
    \end{corollary}
    Note that the variance bound is improved in a similar manner and is omitted here.
    
    
    \section{Discussion}
    \label{sec:discussion}
    
    In this paper, we propose a novel class of HOIF estimators, stable HOIF (sHOIF) estimators, for the doubly robust functionals (DRFs) characterized in \citet{rotnitzky2021characterization}. They are semiparametric efficient under the minimal \Holder{}-smoothness condition $\frac{s_{a} + s_{b}}{2} > \frac{d}{4}$ of \citet{robins2009semiparametric}, allowing the dimension $k$ of the basis function diverging at a rate just slower than the sample size $n$. As can be seen from Theorem \ref{thm:properties}, sHOIF estimators have improved rate of convergence than eHOIF developed in \citet{liu2017semiparametric}. More importantly, as well documented in the simulation studies of \citet{liu2020nearly} and \citet{wanis2023machine}, the sHOIF estimators also have significantly better finite-sample performance over existing higher-order estimators in practice, making them more amenable for tasks such as testing if the bias of a first-order DML estimator $\hat{\psi}_{1}$ of a causal effect is dominated by its standard error \citep{liu2020nearly, liu2021assumption, wanis2023machine}. Finally, we end our paper by mentioning several future research directions:
    \begin{enumerate}[label = (\arabic*)]
        \item It will be interesting to study if one can extend the idea of sHOIF estimators to the non-$\sqrt{n}$-estimable regimes by, for instance, estimating $\Omega$ via some shrinkage or regularized algorithms. As conjectured in \citet{robins2016technical}, the minimax convergence rate of the functionals studied in this paper may depend on the regularity of the density of the covariates $X$. Hence it is expected that the shrinkage or regularization also depends on the density of $X$. Simulation studies in \citet{liu2020nearly} suggest the nonlinear shrinkage covariance matrix estimators of \citet{ledoit2012nonlinear} could be a viable option. Preliminary simulation studies in \citet{liu2020nearly} and \citet{wanis2023machine} suggest that the performance of these shrinkage covariance matrix estimators does degrade with the smoothness of the design density.

        \item As pointed out in \citet{kennedy2022minimax}, their Second-Order R-Learner (SORL) for CATE also involves inverting large Gram matrices of certain basis functions (in which they use the Legendre polynomials) under additional complexity-reducing assumptions on the covariates $X$. It will be interesting to investigate if the sHOIF estimators can be generalized to the CATE estimation problems and stabilize their SORL or even HORL estimators.
        
        \item Another important open problem was also mentioned in \citet{van2014higher, liu2020nearly, liu2021assumption}. To define HOIFs for DRFs, one needs to choose a set of $k$-dimensional basis functions $\sfzbar_{k}$ or an approximation kernel $K$ of the Kronecker delta function, ideally in prior to the data analysis. However, such a strategy seems to go against the current data analytic paradigm, which strongly advocates learning representations (e.g. in the form of bases or kernels) adaptively from data rather than choosing some fixed bases/frames {\it a priori}. Prominent examples include DNNs, autoencoders, and GANs. It is thus interesting to construct HOIF estimators along different basis directions and then select one or aggregate all, guided by certain optimality criterion. We leave this important and difficult problem to future endeavor.
        
        \item It will be interesting to also derive HOIFs and sHOIFs for identifiable causal effect functionals in graphical models with latent variables \citep{bhattacharya2022semiparametric} and implicitly defined functionals \citep{robins2016technical, ai2021unified} in general semiparametric regression problems for improved quality of estimation and statistical inference, which however requires extension of the current work to $U$-processes, a much more difficult research problem that we are working on in a separate paper.
    \end{enumerate}
    
    \bibliographystyle{plainnat}
    \bibliography{\myreferences}
    
    \newpage
    
    \begin{appendices}

    \section{Proof of the variance part of Proposition \ref{prop:soif}}
    \label{app:soif_var}
    
    \subsection{Proof of Lemma \ref{lem:nc}}
    \label{app:nc}
    
    We prove the second statement \eqref{nc-bad} first.
    \begin{align*}
        & \ \left\vert \bbE \left[ \sfzbar_{k} (X_{1})^{\top} \sfzbar_{k} (X_{2}) \sfzbar_{k} (X_{3})^{\top} Q_{3, 4}^{j} \sfzbar_{k} (X_{4}) \right] \right\vert \\
        = & \ \left\vert \bbE [\sfzbar_{k} (X_{1})]^{\top}  \bbE [\sfzbar_{k} (X_{2})] \bbE \left[ \sfzbar_{k} (X_{3})^{\top} Q_{3, 4}^{j} \sfzbar_{k} (X_{4}) \right] \right\vert \\
        \lesssim & \ \left\vert \bbE \left[ \sfzbar_{k} (X_{3})^{\top} Q_{3, 4}^{j} \sfzbar_{k} (X_{4}) \right] \right\vert \\
        \leq & \left\{ \bbE \left[ \sfzbar_{k} (X_{4})^{\top} Q_{3, 4}^{j} \sfzbar_{k} (X_{3}) \sfzbar_{k} (X_{3})^{\top} Q_{3, 4}^{j} \sfzbar_{k} (X_{4}) \right] \right\}^{1 / 2} \lesssim k^{j}.
    \end{align*}
    
    For the first statement \eqref{nc-good}, since $j_{1} > 0$, there must exist at least one $Q_{1, 2}$ between $\sfzbar_{k} (X_{3})$ and $\sfzbar_{k} (X_{4})$ in \eqref{nc-good}. We conduct induction on $j$. $j = 1$ has been proved in the main text. Suppose \eqref{nc-good} holds for $j - 1$. Then by applying Cauchy-Schwarz inequality, one can easily exhibit \eqref{nc-good} for $j$.
    
    \section{Derivation of \eqref{toif_eb}}
    \label{app:toif_eb}
    
    Recall that the kernel estimation bias $\EB_{2, k} (\hat{\psi}_{1})$ of $\hat{\IIFF}_{2, 2, k} (\hat{\Omega})$ truncated at level $j = 2$ is dominated by
    \begin{align}
        \frac{1}{n^{2}} \bbE \left[ \calE_{a} (\hat{a}; O_{1}) \sfzbar_{k} (X_{1})^{\top} Q_{1, 2}^{2} \sfzbar_{k} (X_{2}) \calE_{b} (\hat{b}; O_{2}) \right] + \frac{1}{n} \bbE \left[ \calE_{a} (\hat{a}; O_{1}) \sfzbar_{k} (X_{1})^{\top} \right] \bbE [Q_{3}^{2}] \bbE \left[ \sfzbar_{k} (X_{2}) \calE_{b} (\hat{b}; O_{2}) \right] \label{bias_part}
    \end{align}
    which equals the sum of \eqref{bias_j2_1} and \eqref{bias_j2_2}.
    
    We further consider the kernel estimation bias of $\hat{\IIFF}_{3, 3, k} (\hat{\Omega})$ as an estimator of \eqref{bias_j1}, the dominating term of $\EB_{2, k} (\hat{\psi}_{1})$ truncated at level $j = 1$:
    \begin{align*}
        & \ \bbE \left[ \hat{\IIFF}_{3, 3, k} (\hat{\Omega}) - \eqref{bias_j1} \right] \\
        = & \ \bbE \left[ \calE_{a} (\hat{a}; O_{1}) \sfzbar_{k} (X_{1})^{\top} \left( Q_{1, 2}^{2} - \hat{\Omega} Q_{1, 2}^{2} \hat{\Omega} \right) \sfzbar_{k} (X_{2}) \calE_{b} (\hat{b}; O_{2}) \right] + O (n^{-1}) \\
        = & - \frac{2}{n^{2}} \bbE \left[ \calE_{a} (\hat{a}; O_{1}) \sfzbar_{k} (X_{1})^{\top} Q_{1, 2}^{2} \sfzbar_{k} (X_{2}) \calE_{b} (\hat{b}; O_{2}) \right] + \mathsf{Rem} + O (n^{-1})
    \end{align*}
    where the remainder term $\mathsf{Rem}$ can be shown to dominated by the first term. Adding this term to \eqref{bias_part}, we conclude that the dominating part of the kernel estimation bias of $\hat{\IIFF}_{(2, 2) \rightarrow (3, 3), k} (\hat{\Omega})$ cancels with $\bbE [\tilde{\hat{\IIFF}}_{4, 4, k}]$.
    
    \section{Proof of the kernel estimation bias bound in Theorem \ref{sec:properties}}
    \label{app:main}
    
    We divide the proof of Theorem \ref{sec:properties} into several steps. First, in Section \ref{app:alternative}, we provide alternative characterization of $m$-th order sHOIFs $\hat{\IIFF}_{(2, 2) \rightarrow (m, m), k} (\hat{\Omega})$ to facilitate the bias control.
    
    \subsection{Alternative characterization of sHOIFs} \label{app:alternative}
    We have the following alternative characterization of $m$-th order sHOIFs, which can be shown by induction:
    \begin{equation} \label{hoif_alternative}
        \begin{split}
        & \ \hat{\IIFF}_{2, 2, k} - \hat{\IIFF}_{(2, 2) \rightarrow (m, m), k} (\hat{\Omega}) \\
        \equiv & \ \bbU_{n, 2} \left[ \calE_{a} (\hat{a}; O_{1}) \sfzbar_{k} (X_{1})^{\top} \left\{ \sum_{j = 0}^{m - 1} (-1)^{j} {m - 1 \choose j} \bbU_{n - 2, j - 1} \left( \hat{\Omega} \prod_{s = 3}^{j + 1} Q_{s} \hat{\Omega} \right) \right\} \sfzbar_{k} (X_{2}) \calE_{b} (\hat{b}; O_{2}) \right] \\
        \equiv & \ \bbU_{n, 2} \left[ \calE_{a} (\hat{a}; O_{1}) \sfzbar_{k} (X_{1})^{\top} \left\{ \sum_{j = 1}^{m - 1} (-1)^{j} {m - 1 \choose j} \bbU_{n - 2, j} \left( \hat{\Omega} \prod_{s = 3}^{j + 1} Q_{s} \hat{\Omega} - \bbI \right) \right\} \sfzbar_{k} (X_{2}) \calE_{b} (\hat{b}; O_{2}) \right]
        \end{split}
    \end{equation}
    where $\hat{\Omega} \prod_{s = 3}^{2} Q_{s} \hat{\Omega}$ and $\hat{\Omega} \prod_{s = 3}^{1} Q_{s} \hat{\Omega}$ are understood to be $\hat{\Omega}$ and $\bbI$, respectively.
    
    Armed with \eqref{hoif_alternative}, we can characterize the kernel estimation bias of $\hat{\IIFF}_{(2, 2) \rightarrow (m, m), k} (\hat{\Omega})$ as follows.
    
    \begin{lemma} \label{lem:faux_cp}
        \begin{align}
            & \ \bbE \left[ \hat{\IIFF}_{2, 2, k} - \hat{\IIFF}_{(2, 2) \rightarrow (m, m), k} (\hat{\Omega}) \right] \nonumber \\
            = & \ \sum_{j = 1}^{m - 1} (-1)^{j} {m - 1 \choose j} \sum_{\ell = 1}^{j - 1} \sum_{\substack{S \subseteq [j - 1], \\ |S| = \ell}} \bbE \left[ \calE_{a} (\hat{a}; O_{m - 1}) \sfzbar_{k} (X_{m - 1})^{\top} \prod_{s = 0}^{j - 1} \left\{ Q_{s} (\hat{\Omega} - \bbI)^{\mathbbm{1} \{s \in S\}} \right\} \sfzbar_{k} (X_{m}) \calE_{b} (\hat{b}; O_{m}) \right] \label{bias_general} \\
            = & \ \sum_{\cp = 1}^{m - 1} \sum_{j = \cp}^{m - 1} (-1)^{j} {m - 1 \choose j} \sum_{\substack{S \subseteq [j - 1], \\ |S| = \cp}} \bbE \left[ \calE_{a} (\hat{a}; O_{m - 1}) \sfzbar_{k} (X_{m - 1})^{\top} \prod_{s = 0}^{j - 1} \left\{ Q_{s} (\hat{\Omega} - \bbI)^{\mathbbm{1} \{s \in S\}} \right\} \sfzbar_{k} (X_{m}) \calE_{b} (\hat{b}; O_{m}) \right]. \label{bias_cp}
        \end{align}
    \end{lemma}
    
    \begin{proof}
        To avoid notation clutter, we introduce the alias term $Q_{0} \coloneqq \bbI$. We also overload the notation $[l] \coloneqq \{0, 1, \cdots, l\}$ for any non-negative integer $l$. We first prove \eqref{bias_general}:
        \begin{align*}
            & \ \bbE \left[ \hat{\IIFF}_{2, 2, k} - \hat{\IIFF}_{(2, 2) \rightarrow (m, m), k} (\hat{\Omega}) \right] \\
            = & \ \sum_{j = 1}^{m - 1} (-1)^{j} {m - 1 \choose j} \bbE \left[ \calE_{a} (\hat{a}; O_{m - 1}) \sfzbar_{k} (X_{m - 1})^{\top} \left( \prod_{s = 0}^{j - 1} (Q_{s} \hat{\Omega}) - \bbI \right) \sfzbar_{k} (X_{m}) \calE_{b} (\hat{b}; O_{m}) \right] \\
            = & \ \sum_{j = 1}^{m - 1} (-1)^{j} {m - 1 \choose j} \sum_{\ell = 1}^{j - 1} \sum_{\substack{S \subseteq [j - 1], \\ |S| = \ell}} \bbE \left[ \calE_{a} (\hat{a}; O_{m - 1}) \sfzbar_{k} (X_{m - 1})^{\top} \prod_{s = 0}^{j - 1} \left\{ Q_{s} (\hat{\Omega} - \bbI)^{\mathbbm{1} \{s \in S\}} \right\} \sfzbar_{k} (X_{m}) \calE_{b} (\hat{b}; O_{m}) \right]
        \end{align*}
        where the second equality follows from decomposing each $\hat{\Omega}$ into $\bbI - (\bbI - \hat{\Omega})$ for $j = 2, \cdots, m - 1$.
        
        For \eqref{bias_cp}, we proceed by reorganizing all the summands in \eqref{bias_general} according to the copy number $\cp$ of $\bbI - \hat{\Omega}$ in the product:
        {\small
        \begin{align*}
            \eqref{bias_general} = & \ \sum_{\cp = 1}^{m - 1} \sum_{j = \cp}^{m - 1} (-1)^{j} {m - 1 \choose j} \sum_{\substack{S \subseteq [j - 1], \\ |S| = \cp}} \bbE \left[ \calE_{a} (\hat{a}; O_{m - 1}) \sfzbar_{k} (X_{m - 1})^{\top} \prod_{s = 0}^{j - 1} \left\{ Q_{s} (\hat{\Omega} - \bbI)^{\mathbbm{1} \{s \in S\}} \right\} \sfzbar_{k} (X_{m}) \calE_{b} (\hat{b}; O_{m}) \right].
        \end{align*}
        }
    \end{proof}
    
    \subsection{Analysis by matrix expansion and combinatorics}
    \label{app:main_bias}
    
    After different terms in the kernel estimation bias are reorganized as in \eqref{bias_cp}, we perform the following expansion of $\bbI - \hat{\Omega}$:
    \begin{equation} \label{matrix_expansion}
        \begin{split}
        \hat{\Omega} - \bbI & = \sum_{j = 1}^{J} (\bbI - \hat{\Sigma})^{j} + (\bbI - \hat{\Sigma})^{J + 1} \hat{\Omega}.
        \end{split}
    \end{equation}
    We denote the above expansion up to $J$-th order as
    \begin{align*}
        [\hat{\Omega} - \bbI]_{J} \coloneqq \sum_{j = 1}^{J} \left( \bbI - \hat{\Sigma} \right)^{j}.
    \end{align*}
    
    We then proceed by collecting different terms together by the copy number $\cp$ of $\hat{\Sigma} - \bbI$, and obtain the following lemma.
    
    \begin{lemma} \label{lem:real_cp}
    With $\bbI - \hat{\Omega}$ replaced by $[\bbI - \hat{\Omega}]_{J}$, \eqref{bias_cp} can be rewritten as $\overset{(m - 1) J}{\underset{\cp = 1}{\sum}} \fM_{\cp}$, where
    \begin{equation} \label{bias_real_cp}
        \fM_{\cp} \coloneqq \sum_{j = 1}^{m - 1} (-1)^{j} {m - 1 \choose j} \sum_{\cp' = 1}^{\cp \wedge j} \sum_{\substack{S \subseteq [j - 1], \\ |S| = \cp'}} \sum_{\substack{\{\ell_{s}, s \in S\} \subseteq \{1, \cdots, J\}^{\cp'}, \\ \underset{s \in S}{\sum} \ell_{s} = \cp}} \bbE \left[ \begin{array}{c} 
        \calE_{a} (\hat{a}; O_{m - 1}) \sfzbar_{k} (X_{m - 1})^{\top} \\
        \times \prod\limits_{s = 0}^{j - 1} \left\{ Q_{s} (\bbI - \hat{\Sigma})^{\ell_{s} \mathbbm{1} \{s \in S\}} \right\} \\
        \times \ \sfzbar_{k} (X_{m}) \calE_{b} (\hat{b}; O_{m})
        \end{array} \right].
    \end{equation}
    \end{lemma}
    
    \begin{proof}
    The proof follows from a few lines of algebra.
    \begin{align*}
        & \ \sum_{\cp' = 1}^{m - 1} \sum_{j = \cp'}^{m - 1} (-1)^{j} {m - 1 \choose j} \sum_{\substack{S \subseteq [j - 1], \\ |S| = \cp'}} \bbE \left[ \calE_{a} (\hat{a}; O_{m - 1}) \sfzbar_{k} (X_{m - 1})^{\top} \prod_{s = 0}^{j - 1} \left\{ Q_{s} [\hat{\Omega} - \bbI]_{J}^{\mathbbm{1} \{s \in S\}} \right\} \sfzbar_{k} (X_{m}) \calE_{b} (\hat{b}; O_{m}) \right] \\
        = & \ \sum_{\cp' = 1}^{m - 1} \sum_{j = \cp'}^{m - 1} (-1)^{j} {m - 1 \choose j} \sum_{\substack{S \subseteq [j - 1], \\ |S| = \cp'}} \bbE \left[ \begin{array}{c}
        \calE_{a} (\hat{a}; O_{m - 1}) \sfzbar_{k} (X_{m - 1})^{\top} \\
        \times \prod\limits_{s = 0}^{j - 1} \left\{ Q_{s} \left( \sum\limits_{\ell_{s} = 1}^{J} (\bbI - \hat{\Sigma})^{\ell_{s}} \right)^{\mathbbm{1} \{s \in S\}} \right\} \\
        \times \ \sfzbar_{k} (X_{m}) \calE_{b} (\hat{b}; O_{m})
        \end{array} \right] \\
        = & \ \sum_{\cp = 1}^{(m - 1) J} \sum_{\cp' = 1}^{(m - 1) \wedge \cp} \sum_{j = \cp'}^{m - 1} (-1)^{j} {m - 1 \choose j} \sum_{\substack{S \subseteq [j - 1], \\ |S| = \cp'}} \sum_{\substack{\{\ell_{s}, s \in S\} \subseteq \{1, \cdots, J\}^{\cp'}, \\ \underset{s \in S}{\sum} \ell_{s} = \cp}} \bbE \left[ \begin{array}{c} 
        \calE_{a} (\hat{a}; O_{m - 1}) \sfzbar_{k} (X_{m - 1})^{\top} \\
        \times \prod\limits_{s = 0}^{j - 1} \left\{ Q_{s} (\bbI - \hat{\Sigma})^{\ell_{s} \mathbbm{1} \{s \in S\}} \right\} \\
        \times \ \sfzbar_{k} (X_{m}) \calE_{b} (\hat{b}; O_{m})
        \end{array} \right] \\
        = & \ \sum_{\cp = 1}^{(m - 1) J} \sum_{j = 1}^{m - 1} (-1)^{j} {m - 1 \choose j} \sum_{\cp' = 1}^{\cp \wedge j} \sum_{\substack{S \subseteq [j - 1], \\ |S| = \cp'}} \sum_{\substack{\{\ell_{s}, s \in S\} \subseteq \{1, \cdots, J\}^{\cp'}, \\ \underset{s \in S}{\sum} \ell_{s} = \cp}} \bbE \left[ \begin{array}{c} 
        \calE_{a} (\hat{a}; O_{m - 1}) \sfzbar_{k} (X_{m - 1})^{\top} \\
        \times \prod\limits_{s = 0}^{j - 1} \left\{ Q_{s} (\bbI - \hat{\Sigma})^{\ell_{s} \mathbbm{1} \{s \in S\}} \right\} \\
        \times \ \sfzbar_{k} (X_{m}) \calE_{b} (\hat{b}; O_{m})
        \end{array} \right].
    \end{align*}
    \end{proof}
    
    To proceed further, we also need the following elementary lemma:
    \begin{lemma}
    \label{lem:expansion}
    Given $j$ nonnegative integers $\ell_{0}, \cdots, \ell_{j - 1} \in \bbZ_{\geq 0}$ and let $\bar{\ell}_{h} \coloneqq \sum\limits_{s = 0}^{h} \ell_{s}$ for $h = 0, \cdots, j - 1$ with $\bar{\ell}_{-1} \equiv 0$,
        \begin{equation}
            \begin{split}
                & \ \bbE \left[ \calE_{a} (\hat{a}; O_{m - 1}) \sfzbar_{k} (X_{m - 1})^{\top} \prod\limits_{s = 0}^{j - 1} \left\{ Q_{s} (\bbI - \hat{\Sigma})^{\ell_{s}} \right\} \sfzbar_{k} (X_{m}) \calE_{b} (\hat{b}; O_{m}) \right] \\
                = & \ n^{- \bar{\ell}_{j - 1}} \sum_{i_{1} = 1}^{n} \cdots \sum_{i_{\bar{\ell}_{j - 1}} = 1}^{n} \bbE \left[ \calE_{a} (\hat{a}; O_{m - 1}) \sfzbar_{k} (X_{m - 1})^{\top} \left\{ \prod_{s = 0}^{j - 1} Q_{s} \prod_{h = \bar{\ell}_{s - 1} + 1}^{\bar{\ell}_{s}} (\bbI - Q_{i_{h}}) \right\} \sfzbar_{k} (X_{m}) \calE_{b} (\hat{b}; O_{m}) \right].
            \end{split}
        \end{equation}
    \end{lemma}

    \begin{proof}
    Without essential loss of generality, we take $\calE_{a} (\hat{a}; O) \equiv A, \calE_{b} (\hat{b}; O) \equiv 1$ to simplify the exposition. Repeatedly invoking the identity $\bbI - \hat{\Sigma} \equiv n^{-1} \sum_{i = 1}^{n} (\bbI - Q_{i})$, together with the convention $\prod_{h = \ell + 1}^{\ell} (\cdot)_{h} \equiv \bbI$ for any nonnegative integer $\ell$, we have
    \begin{align*}
        & \ \bbE \left[ A_{m - 1} \sfzbar_{k} (X_{m - 1})^{\top} \prod\limits_{s = 0}^{j - 1} \left\{ Q_{s} (\bbI - \hat{\Sigma})^{\ell_{s}} \right\} \sfzbar_{k} (X_{m}) \right] \\
        = & \ n^{- \bar{\ell}_{j - 1}} \bbE \left[ A_{m - 1} \sfzbar_{k} (X_{m - 1})^{\top} \left\{ \prod_{s = 0}^{j - 1} Q_{s} \prod_{h = 1}^{\ell_{s}} \left[ \sum_{i_{s, h} = 0}^{n} (\bbI - Q_{i_{s, \ell_{s}}}) \right] \right\} \sfzbar_{k} (X_{m}) \right] \\
        = & \ n^{- \bar{\ell}_{j - 1}} \sum_{i_{1} = 1}^{n} \cdots \sum_{i_{\bar{\ell}_{j - 1}} = 1}^{n} \bbE \left[ A_{m - 1} \sfzbar_{k} (X_{m - 1})^{\top} \left\{ \prod_{s = 0}^{j - 1} Q_{s} \prod_{h = \bar{\ell}_{s - 1} + 1}^{\bar{\ell}_{s}} (\bbI - Q_{i_{h}}) \right\} \sfzbar_{k} (X_{m}) \right].
    \end{align*}
    \end{proof}

    With the above preparatory steps, the following ``cancellation lemma'' is the first key milestone towards completing the proof.

    \begin{lemma}
    \label{cancellation lemma}
    For copy numbers satisfying $\cp < \lceil (m - 1) / 2 \rceil$, $$\mathrm{Equation \ } \eqref{bias_real_cp} \equiv 0.$$
    \end{lemma}

    \begin{proof}
    Again, without loss of generality, we take $\calE_{a} (\hat{a}; O) \equiv A, \calE_{b} (\hat{b}; O) \equiv 1$. Aided by Lemma \ref{lem:expansion}, the summand in Equation \eqref{bias_real_cp} at any given $\cp$ can be rewritten as
    \begin{align}
        & \ \sum_{j = 1}^{m - 1} (-1)^{j} \binom{m - 1}{j} \sum_{\cp' = 1}^{\cp \wedge j} \sum_{\substack{S \subseteq [j - 1], \\ |S| = \cp'}} \sum_{\substack{\{\ell_{s}, s \in S\} \subseteq \{1, \cdots, J\}^{\cp'}, \\ \sum\limits_{s \in S} \ell_{s} = \cp}} \bbE \left[ A_{m - 1} \sfzbar_{k} (X_{m - 1})^{\top} \prod\limits_{s = 0}^{j - 1} \left\{ Q_{s} (\bbI - \hat{\Sigma})^{\ell_{s} \mathbbm{1} \{s \in S\}} \right\} \sfzbar_{k} (X_{m}) \right] \label{key expansion 2} \\
        = & \ n^{- \cp} \sum_{i_{1}, \cdots, i_{\cp} = 1}^{n} \sum_{j = 1}^{m - 1} \binom{m - 1}{j} \sum_{\cp' = 1}^{\cp \wedge j} \sum_{\substack{S \subseteq [j - 1], \\ |S| = \cp'}} \sum_{\substack{\{\ell_{s}, s \in S\} \subseteq \{1, \cdots, J\}^{\cp'}, \\ \sum\limits_{s \in S} \ell_{s} = \cp}} \bbE \left[ \begin{array}{c}
        A_{m - 1} \sfzbar_{k} (X_{m - 1})^{\top} \\
        \times \prod\limits_{s = 0}^{j - 1} (- \ Q_{s}) \prod\limits_{h = \bar{\ell}_{s - 1} + 1}^{\bar{\ell}_{s}} (\bbI - Q_{i_{h}}) \\
        \times \ \sfzbar_{k} (X_{m})
        \end{array} \right] \notag \\
        = & \ n^{- \cp} \sum_{i_{1}, \cdots, i_{\cp} = 1}^{n} \sum_{j = 1}^{m - 1} \binom{m - 1}{j} \sum_{\cp' = 1}^{\cp \wedge j} \sum_{\substack{S \subseteq [j - 1], \\ |S| = \cp'}} \sum_{\substack{\{\ell_{s}, s \in S\} \subseteq \{1, \cdots, J\}^{\cp'}, \\ \sum\limits_{s \in S} \ell_{s} = \cp}} \bbE \left[ \begin{array}{c}
        A_{m - 1} \sfzbar_{k} (X_{m - 1})^{\top} \\
        \times \prod\limits_{s = 0}^{j - 1} \left\{ \begin{array}{c} 
        (\bbI - Q_{s}) \prod\limits_{h = \bar{\ell}_{s - 1} + 1}^{\bar{\ell}_{s}} (\bbI - Q_{i_{h}}) \\
        - \prod\limits_{h = \bar{\ell}_{s - 1} + 1}^{\bar{\ell}_{s}} (\bbI - Q_{i_{h}})
        \end{array} \right\} \\
        \times \ \sfzbar_{k} (X_{m})
        \end{array} \right]. \label{key expansion}
    \end{align}
    Now we introduce another auxiliary copy number $\cp^{\dag}$, collecting all the terms in the above display with $\cp^{\dag}$ many $\bbI - \bbQ_{s}$'s after expanding the following product
    \begin{equation}
    \prod\limits_{s = 0}^{j - 1} \left\{ (\bbI - Q_{s}) \prod\limits_{h = \bar{\ell}_{s - 1} + 1}^{\bar{\ell}_{s}} (\bbI - Q_{i_{h}}) - \prod\limits_{h = \bar{\ell}_{s - 1} + 1}^{\bar{\ell}_{s}} (\bbI - Q_{i_{h}}) \right\} \label{product}
    \end{equation}
    within the expectation of \eqref{key expansion}. Let $\mathsf{p} (n, k)$ be the number of all possible partitions of $n$ into $k$ positive integers. Upon expanding, for any given $\cp^{\dag}$, the expectations are all of the following form:
    \begin{equation}
    \label{key expectation}
    \bbE \left[ A_{m - 1} \sfzbar_{k} (X_{m - 1})^{\top} \prod_{\substack{s \in S \subseteq [j - 1]}, \\ |S| = \cp^{\dag}} (\bbI - Q_{s}) \prod_{l = 1}^{\cp} (\bbI - Q_{i_{l}}) \sfzbar_{k} (X_{m}) \right]
    \end{equation}
    up to permuting the orders of different $(\bbI - Q_{s})$ and $(\bbI - Q_{i_l})$. Hence the coefficient constant of the corresponding expectation (again, up to permutations) is
    \begin{align}
        & \ \sum_{j = 1}^{m - 1} \binom{m - 1}{j} \left\{ \sum_{\ell = 1}^{\cp \wedge j} \binom{j}{\ell} \mathsf{p} (\cp, \ell) \right\} (-1)^{j - \cp^{\dag}} \binom{j}{\cp^{\dag}} \notag \\
        = & \ (-1)^{- \cp^{\dag}} \sum_{j = 0}^{m - 1} (-1)^{j} \binom{m - 1}{j} \binom{j + \cp - 1}{j - 1} \binom{j}{\cp^{\dag}} \notag \\
        = & \ \frac{(-1)^{- \cp^{\dag}}}{\cp^{\dag}! \cp!} \sum_{j = 0}^{m - 1} (-1)^{j} \binom{m - 1}{j} j \prod_{\ell = - \cp^{\dag} + 1}^{\cp - 1} (j + \ell) \label{cancel}
    \end{align}
    where the first equality follows from Lemma S7. The coefficient constant is simply counting the number of terms after expanding \eqref{key expansion}: in the first line of the above display, $\binom{m - 1}{j}$ comes from $\binom{m - 1}{j}$ of \eqref{key expansion}, $\left\{ \sum_{\cp' = 1}^{\cp \wedge j} \binom{j}{\cp'} \mathsf{p} (\cp, \cp') \right\}$ arises from the three summations after $\binom{m - 1}{j}$ of \eqref{key expansion}, and $(-1)^{j - \cp^{\dag}} \binom{j}{\cp^{\dag}}$ counts the number of terms with $\cp^{\dag}$ many $(\bbI - Q_{s})$'s, for $s \in \{0, 1, \cdots, j - 1\}$, upon expanding \eqref{product}.
    
    Another key observation is that after expansion, the expectation of \eqref{key expansion} is identically zero when $\cp^{\dag} > \cp$, leading to zero summands regardless of its coefficient constant. By virtue of this observation, we only need to consider the case when $\cp^{\dag} \leq \cp$. It takes elementary calculations to show there exists integers $\gamma_{1}, \cdots, \gamma_{\cp + \cp^{\dag} - 1}$ such that
    \begin{align*}
    \eqref{cancel} = \frac{(-1)^{- \cp^{\dag}}}{\cp^{\dag}! \cp!} \sum_{j = 0}^{m - 1} (-1)^{j} \binom{m - 1}{j} \left( j^{\cp + \cp^{\dag}} + \sum_{\ell = 1}^{\cp + \cp^{\dag} - 1} \gamma_{\ell} j^{\ell} \right).
    \end{align*}
    Hence when 
    \begin{equation}
    \label{cancel condition}
    m - 1 - (\cp + \cp^{\dag}) > 0,
    \end{equation}
    $\eqref{cancel} \equiv 0$ by differentiating the binomial identity as in Lemma S8. Since we have assumed that $\cp \geq \cp^{\dag}$, $\cp < \lceil (m - 1) / 2 \rceil$ suffices for \eqref{cancel condition} to hold. This concludes the proof.
    \end{proof}

    Following Lemma \ref{cancellation lemma}, the next important observation wraps up the proof.

    \begin{lemma}
    \label{well control lemma}
    For copy numbers satisfying $\cp \geq \lceil (m - 1) / 2 \rceil$,
    \begin{align*}
        |\mathrm{Equation \ } \eqref{bias_real_cp}| \leq \left( \frac{k m}{n} \right)^{\lceil \frac{\cp - 1}{2} \rceil \vee 1}.
    \end{align*}
    \end{lemma}

    \begin{proof}
    This proof inherits the notations defined in the proof of Lemma \ref{cancellation lemma}. We consider the case $\cp \geq \lceil (m - 1) / 2 \rceil$, which, as shown in the previous lemma, is not identically zero. We need to count the number of non-zero expectations, which is easier to work out using the representation \eqref{key expansion 2}.

    For any given copy number $\cp$, the number of non-zero expectations of the form \eqref{key expectation} is
    \begin{align}
    \sum_{j = 1}^{m - 1} (-1)^{j} \binom{m - 1}{j} \binom{j + \cp - 1}{\cp} = (-1)^{m - 1} \binom{\cp - 1}{(\cp - m + 1) \vee 0} \label{control comb}
    \end{align}
    by employing Lemma S9 in the online supplements. When $\cp < m - 1$, $\eqref{control comb} = (-1)^{m - 1}$; whereas when $\cp \geq m - 1$:
    \begin{itemize}
    \item if $m \geq 4$ and $\cp \geq 5$
    \begin{align*}
    |\eqref{control comb}| & = \binom{\cp - 1}{m - 2} \leq (\cp - 1)^{m - 2} \leq m^{\cp};
    \end{align*}
    \item if $m = 4$ and $\cp = 4$,
    \begin{align*}
    |\eqref{control comb}| = \binom{3}{2} = 3 < m^{\cp};
    \end{align*}
    \item finally, if $m = 3$,
    \begin{align*}
    |\eqref{control comb}| = \cp - 1 < m^{\cp}.
    \end{align*}
    \end{itemize}

    The proof is completed by bounding the absolute value of these non-zero expectations of the form \eqref{key expectation} by $k^{\cp}$ using Lemma S1, leading to the claim $$|\eqref{bias_real_cp}| \leq \left( \frac{k m}{n} \right)^{\lceil \frac{\cp - 1}{2} \rceil \vee 1}.$$
    \end{proof}

    \begin{remark}
    It is possible to improve the upper bound $m^{\cp}$ of $|\eqref{control comb}|$. First, the upper bound for the binomial coefficient used here is not sharp. Second, not all the terms counted in \eqref{control comb} (i.e. the term $\binom{j + \cp - 1}{\cp}$) are nonzero. We decide not to pursue an improvement over $m^{\cp}$ for aesthetic purpose.
    \end{remark}

    Finally, combining the above results, we have the desired kernel estimation bias bound given in Theorem \ref{thm:properties}.
    
    \end{appendices}

\end{document}